\documentclass{amsart}
\usepackage{amssymb,euscript,amsmath, mathrsfs, stmaryrd, textcomp}
\usepackage{comment}
\usepackage[pdftex]{graphicx}
\usepackage[pdftex]{color}
\definecolor{navy}{rgb}{0,0,.5}
\definecolor{darkred}{rgb}{1,0,0}
\usepackage{pictexwd,dcpic}
\usepackage{amscd}
\usepackage{enumerate}

\input xy
\xyoption{all}
\CompileMatrices

\newcommand{\llb}{\ensuremath{\llbracket}}
\newcommand{\rrb}{\ensuremath{\rrbracket}}

\newcommand{\llpar}{(\negthinspace(}
\newcommand{\rrpar}{)\negthinspace)}

\def\KK{K}
\def\Ka{\overline{K}}
\def\Ra{\overline{R}}
\def\ka{\overline{k}}

\def\CC{{\mathbb C}}

\def\RR{{\mathbb R}}
\def\R{{\mathbb R}}
\def\ZZ{{\mathbb Z}}
\def\Z{{\mathbb Z}}
\def\QQ{{\mathbb Q}}
\def\Q{{\mathbb Q}}
\def\AA{{\mathbb A}}
\def\PP{{\mathbb P}}
\def\TT{{\mathbb T}}
\def\GG{{\mathbb G}}

\def\red{\operatorname{red}}
\def\Spec{\operatorname{Spec}}

\def\Hom{\operatorname{Hom}}

\def\trop{\operatorname{trop}}

\def\Star{\operatorname{Star}}

\def\cU{{\mathcal U}}
\def\cV{{\mathcal V}}
\def\cW{{\mathcal W}}
\def\cX{{\mathcal X}}
\def\cY{{\mathcal Y}}
\def\cZ{{\mathcal Z}}

\def\fU{{\mathfrak U}}

\def\fX{{\mathfrak X}}
\def\fY{{\mathfrak Y}}
\def\fZ{{\mathfrak Z}}

\def\<{{\langle}}
\def\>{{\rangle}}

\DeclareMathOperator{\val}{val}
\DeclareMathOperator{\Gal}{Gal}

\newcommand{\ignore}[1]{}

\numberwithin{equation}{subsection}
\newtheorem{thm}[equation]{Theorem}
\newtheorem{theorem}[equation]{Theorem}
\newtheorem{prop}[equation]{Proposition}
\newtheorem{lem}[equation]{Lemma}
\newtheorem{lemma}[equation]{Lemma}
\newtheorem{cor}[equation]{Corollary}

\newtheorem*{thm*}{Theorem}

\newtheorem*{conjA}{Conjecture A}
\newtheorem*{conjB}{Conjecture B}

\theoremstyle{definition}

\newtheorem{definition}[equation]{Definition}

\newtheorem{exam}[equation]{Example}

\theoremstyle{remark}

\newtheorem{rem}[equation]{Remark}

\newcommand{\spe}{\mathrm{sp}}
\newcommand{\VF}{\mathrm{VF}}
\newcommand{\RES}{\mathrm{RES}}
\newcommand{\Var}{\mathrm{Var}}
\newcommand{\Vol}{\mathrm{Vol}}

\newcommand{\gal}{\widehat{\mu}}
\newcommand{\gro}{\mathbf{K}}

\def\Q{\mathbb{Q}}
\def\LL{{\mathbb L}}

\begin{document}
\title[A tropical motivic Fubini theorem]{A tropical motivic Fubini theorem with applications to Donaldson-Thomas theory}

\author{Johannes Nicaise}
\author{Sam Payne}
\address{Johannes Nicaise, Imperial College,
Department of Mathematics, South Kensington Campus,
London SW72AZ, UK, and KU Leuven, Department of Mathematics, Celestijnenlaan 200B, 3001 Heverlee, Belgium} \email{j.nicaise@imperial.ac.uk}
\address{Sam Payne, Mathematics Department, Yale University, New Haven, CT 06511,USA}
\email{sam.payne@yale.edu}

\begin{abstract}
We present a new tool for the calculation of Denef and Loeser's motivic nearby fiber and motivic Milnor fiber: a motivic Fubini theorem for the tropicalization map, based on Hrushovski and Kazhdan's theory of motivic volumes of semi-algebraic sets. As applications, we prove a conjecture of Davison and Meinhardt on motivic nearby fibers of weighted homogeneous polynomials, and give a very short and conceptual new proof of the integral identity conjecture of Kontsevich and Soibelman, first proved by L\^e Quy Thuong.  Both of these conjectures emerged in the context of motivic Donaldson-Thomas theory.
 \end{abstract}

\maketitle

\section{Introduction}

Let $k$ be a field of characteristic zero that contains all roots of unity.
Denef and Loeser's motivic nearby fiber, motivic vanishing cycles, and motivic Milnor fiber are subtle invariants of hypersurface singularities over $k$.
 They were defined as elements of $\mathcal{M}^{\gal}_k=\gro^{\gal}(\Var_k)[\LL^{-1}]$, the Grothendieck ring of $k$-varieties with an action of the profinite group $\gal$ of roots of unity, localized
 with respect to the class $\LL$ of the affine line (or a suitable relative variant of this ring) \cite[3.5.3]{DL}.
 These invariants should be viewed as motivic incarnations of the
 nearby and vanishing cycles complexes and the topological Milnor fiber, respectively, where the $\gal$-action reflects the monodromy.
 They play a central role in various applications in birational geometry and singularity theory, for instance
 in the calculation of the Hodge spectrum \cite{GLM}. They are also central tools in motivic Donaldson-Thomas theory, where the motivic vanishing cycles and the motivic Milnor fiber appear as geometric upgrades of the virtual Euler characteristic and the Behrend function.

\subsection{The tropical motivic Fubini theorem}  We present a new tool for the calculation of these invariants, based on tropical geometry and Hrushovski and Kazhdan's theory of motivic integration \cite{HK}.  The motivic vanishing cycles are defined by subtracting the class of the hypersurface from the motivic nearby fiber, so we restrict our attention to the other two invariants. The theory of Hrushovski and Kazhdan assigns to every semi-algebraic set $S$ over the field $K_0=k\llpar t \rrpar$ a motivic volume $\Vol(S)$ in $\gro^{\gal}(\Var_k)$, and we give natural interpretations of the motivic nearby fiber and the motivic Milnor fiber as motivic volumes of semi-algebraic sets (Corollary~\ref{cor:compar}).    One advantage of this approach is that the invariants are well-defined already in $\gro^{\gal}(\Var_k)$, without inverting $\LL$ (see Remark \ref{rem:compar}).  Another more striking advantage is that we can use semi-algebraic decompositions of these semi-algebraic sets to compute their motivic volumes, and thereby exploit natural connections to tropical geometry.  In particular, we present a new method to compute such motivic volumes: a motivic Fubini theorem for the tropicalization map (Theorem~\ref{thm:fubini}), which we state as follows.

\begin{thm*}
 Let $Y$ be a $\KK_0$-variety.
Let $n$ be a positive integer and let $S$ be a semi-algebraic subset of $\mathbb{G}^n_{m,\KK_0}\times_{\KK_0}Y$.
 Denote by $$\pi:\mathbb{G}^n_{m,\KK_0}\times_{\KK_0}Y\to \mathbb{G}^n_{m,\KK_0}$$ the projection morphism.
 Then the function
 $$(\trop \circ \pi)_*\mathbf{1}_S:\Q^n\to \gro^{\gal}(\Var_k):w\mapsto \Vol(S\cap (\trop\circ \pi)^{-1}(w))$$
 is constructible, and
 $$\Vol(S)=\int_{\Q^n}(\trop\circ \pi)_*\mathbf{1}_S \,d\chi'$$ in
 $\gro^{\gal}(\Var_k)$.
\end{thm*}
\noindent Here $\trop$ is the tropicalization map and $\chi'$ is the so-called bounded Euler characteristic, the unique additive invariant
 on the Boolean algebra generated by polyhedra in $\Q^n$ that assigns the value $1$ to every closed polyhedron.
 In many situations, one can show that the function $(\trop \circ \pi)_*\mathbf{1}_S$ is constant
 on polyhedral subsets of $\Q^n$ with trivial bounded Euler characteristic. The Fubini theorem then allows us to discard
 the contribution of these pieces to $\Vol(S)$ without knowing anything about the motivic volumes of the fibers of the tropicalization map, which may be difficult to control.

This approach seems to be surprisingly effective.  Indeed, in each of our applications below, we prove the desired identity of motivic volumes by first giving an inclusion of semi-algebraic sets.  We then tropicalize the complement and use a $\GG_m$-action to show that the volumes of the fibers are constant on polyhedral subsets of $\Q^n$ with trivial bounded Euler characteristic.

\subsection{Applications}
We use our motivic Fubini theorem to solve the Davison-Meinhardt conjecture on motivic nearby fibers of weighted homogeneous polynomials \cite[5.5]{DM}.
We also give a very short and conceptual proof of the integral identity conjecture of Kontsevich and Soibelman \cite[\S4.4]{KS}, which was first proved by L\^e Quy Thuong in \cite{thuong}. Both of these conjectures emerged in motivic Donaldson-Thomas theory; let us recall their statements.

\begin{conjA}[Davison-Meinhardt, 2011]
Let $Y$ be a smooth $k$-variety with the trivial $\GG_{m,k}$-action, and let $\GG_{m,k}$ act on $\AA^n_k$ with weights $w_1,\ldots,w_n> 0$.
 Let $$f:\AA^n_k\times_k Y\to \AA^1_k$$ be a $\GG_{m,k}$-equivariant function, where $\GG_{m,k}$ acts on $\AA^1_k$ with weight $d>0$.
 Then the motivic nearby fiber of $f$ is equal to $[f^{-1}(1)]$ in $\mathcal{M}^{\gal}_k$, where the $\gal$-action on $f^{-1}(1)$ factors through $\mu_d(k)$ and is given by
      $$\mu_d(k)\times f^{-1}(1)\to f^{-1}(1):(\zeta,(x_1,\ldots,x_n,y))\mapsto (\zeta^{w_1}x_1,\ldots,\zeta^{w_n}x_n,y).$$
\end{conjA}

Our formulation is
equivalent with the one in \cite[5.5]{DM} except that they ask for an equality in the localized Grothendieck ring of varieties over the base $Y$; we will discuss this refinement in Proposition \ref{prop:DMrel}.
  Davison and Meinhardt proved their conjecture in the special case where $w_i=1$ for all $i$ \cite[5.9]{DM-quiver}, extending an earlier result of Behrend, Bryan and Szendr{\H o}i, who handled the case where also $d=1$ \cite[2.12]{BBS}.  We prove the general case in Theorem~\ref{thm:DM}.  Our argument also yields a natural generalization, in which $\AA^n_k$ is replaced by a $\GG_{m,k}$-invariant subvariety of a circle compact toric variety.  See Theorem~\ref{thm:genDM}.

The following statement was conjectured by Kontsevich and Soibelman, who described it as crucial to their theory of motivic Donaldson-Thomas invariants \cite[\S4.4]{KS}.
\begin{conjB}[Kontsevich-Soibelman, 2008]
Let $d_1$, $d_2$ and $d_3$ be nonnegative integers and let $\GG_{m,k}$ act diagonally on $$U=\AA^{d_1}_k\times_k \AA^{d_2}_k\times_k \AA^{d_3}_k$$
with weight $1$ on the first factor, with weight $-1$ on the second factor, and trivially on the third factor.
 Let $$f:U\to \AA^1_k$$ be a $\GG_{m,k}$-equivariant function, where $\GG_{m,k}$ acts trivially on the target $\AA^1_k$, and such that $f(0,0,0)=0$.
 Denote by $f|_{\AA^{d_3}_k}$ be the restriction of $f$ to $\AA^{d_3}_k$ {\em via} the embedding $\AA^{d_3}_k\to U:z\mapsto (0,0,z)$.
  Then the restriction of the motivic nearby fiber of $f$ to $\AA^{d_1}_k\subset f^{-1}(0)$ equals $\LL^{d_1}$ times the motivic Milnor fiber of $f|_{\AA^{d_3}_k}$ at $0$,
 where we view both objects as elements of
   $\mathcal{M}^{\gal}_k$.
\end{conjB}
 \noindent This statement, widely known as the integral identity conjecture, was proved by L\^e Quy Thuong \cite{thuong}, and his proof also  uses Hrushovski-Kazhdan motivic integration.  Our tropical motivic Fubini theorem allows us to substantially
  simplify the proof, and to generalize it in the following way: we allow
  arbitrary positive weights on $\AA^{d_1}_k$ and arbitrary negative weights on $\AA^{d_2}_k$, and we replace the factor $\AA^{d_3}_k$ by any $k$-variety with trivial
  $\GG_{m,k}$-action.  See Theorem~\ref{thm:KS}.  Our argument also gives a further generalization, in which $\AA^{d_1}_k$ is replaced by a connected $\GG_{m,k}$-invariant subvariety of a circle compact toric variety, and $\AA^{d_2}_k$ is replaced by an affine toric variety with repelling fixed point.  See Theorem~\ref{thm:genKS}.

  In all of these applications, we prove equalities in $\gro^{\gal}(\Var_k)$, without inverting $\LL$.  The resulting statements are stronger, because $\LL$ is a zero-divisor \cite{borisov}.

\subsection{Plan of the paper}
In Section \ref{sec:motvol}, we explain the construction of the motivic volume of Hrushovski and Kazhdan, which is based on the model theory of algebraically closed valued fields.  We have made an effort to present the results in geometric terms in Theorem~\ref{thm:HK}.
We have given a similar presentation in \cite{NPS} in the setting where the base field
is an algebraic closure of $K_0$, rather than $K_0$ itself.  For the applications in this paper it is essential to keep track of the Galois action of $\gal$ on the motivic volume, which requires the more careful analysis given here.  We then prove an explicit formula for the motivic volume in terms of a
 strict normal crossings model (Theorem~\ref{thm:snc}), and show how to realize the motivic nearby fiber and motivic Milnor fiber as motivic volumes of semi-algebraic sets (Corollary~\ref{cor:compar}).

  In Section \ref{sec:fubini} we prove the tropical motivic Fubini theorem (Theorem~\ref{thm:fubini}). The proof proceeds in two steps:
  we first consider the sch{\"o}n case, where one can construct an explicit semi-algebraic decomposition of the semi-algebraic set into elementary pieces.
  We then use a result of Luxton and Qu \cite{LQ} to decompose an arbitrary semi-algebraic set into sch{\"o}n pieces.
    In Section \ref{sec:appli}, we present two applications: the proofs of Conjectures A and B (Theorems \ref{thm:DM} and \ref{thm:KS}). Finally, in Section \ref{sec:general}, we explain how to refine the constructions relatively over a base scheme, and how to generalize Conjectures A and B to invariant subvarieties of toric varieties equipped with a $\GG_m$-action.

\subsection{Notation}\label{ss:notation} Let $k$ be field of characteristic zero that contains all roots of unity.\footnote{The condition that $k$ contain all roots of unity is not strictly necessary; we include it to make $\gal$ a profinite group, rather than merely a profinite group scheme over $k$. This avoids additional descent arguments in the proofs. However, we do not want to assume that $k$ is algebraically closed: even if one is ultimately interested in the case $k=\CC$, one needs to consider finitely generated extensions of $\CC$ to study relative motivic invariants over a base variety in Section \ref{ssec:rel}.} We set $\KK_0=k\llpar t\rrpar$ and $R_0=k\llb t \rrb$.
We denote by $\KK$ the field of Puiseux series $$\KK=\bigcup_{n>0}k\llpar t^{1/n} \rrpar$$
 and we fix an algebraic closure $\Ka$ of $\KK$.
  The $t$-adic valuation on $\KK_0$ extends uniquely to a valuation
 $$\val:\Ka^{\times}\to \Q$$ on $\Ka$. We further extend it to $\Ka$ by setting $\val(0)=\infty$, and we extend the natural order on $\Q$ to $\Q\cup\{\infty\}$
 by declaring that $q\leq \infty$ for all $q$ in $\Q\cup \{\infty\}$.
 We write $R$ for the valuation ring in $\KK$, and $\widehat{R}$ for its $t$-adic completion. We also write $\Ra$ for the valuation ring in $\Ka$; its residue field is an algebraic closure $\ka$ of $k$.

The Galois group $\mathrm{Gal}(\KK/\KK_0)$ is canonically isomorphic to the profinite group
$$\gal=\lim_{\stackrel{\longleftarrow}{n}}\mu_n(k)$$
of roots of unity in $k$. This isomorphism also defines a splitting of the short exact sequence
 $$1\to \gal\to \Gal(\Ka/\KK_0)\to \Gal(\ka/k)\to 1.$$
The group $\gal$ acts continuously on $\Ka$ from the left by means of the rule: $\zeta \ast t^{1/n}=\zeta t^{1/n}$ for $\zeta\in \mu_n(k)$, $n>0$.
 We will consider the {\em inverse} right action of $\gal$ on $\Ka$ so that $\gal$ acts on $\Spec \Ka$ and $\Spec \Ra$ from the {\em left}.
 This convention will be important for the comparison results in Section \ref{sec:DL}.

 Let $M$ be a free $\ZZ$-module of finite rank and let $\TT$ be the split $R$-torus with character lattice $M$.
   Then we can consider the tropicalization map
  $$\trop:\TT(\Ka)\to \Hom(M,\QQ):x\mapsto \left(m\mapsto \val(\chi^m(x))\right).$$
  Let $w$ be an element of $\Hom(M,\Q)$, and write $w=v/d$ for some positive integer $d$ and some element $v$ in $M^{\vee}$.
 Consider the left $\mu_d(k)$-action on $\TT_{k}$ with weight vector $v$; that is,
 each element $\zeta$ in $\mu_d(k)$ acts on the character $\chi^m$ by multiplication with $\zeta^{\langle v,m\rangle}$, for every $m\in M$.
 This induces a left $\gal$-action
 on $\TT_{k}$ that we call the $\gal$-action with weight vector $w$. The $k$-variety $\TT_{k}$ endowed with this action will be denoted by
 $\TT^w_{k}$, and we will write $\TT^w$ and $\TT^w_{\KK}$ for the varieties $\TT^w_{k}\times_k R$ and $\TT^w_{k}\times_k K$ endowed with the diagonal $\gal$-actions.
 Then multiplication with $t^{\langle w,\cdot\rangle}$ defines a $\gal$-equivariant bijection between $\TT^w(\Ra)=\Hom(M,\Ra^{\times})$ and $\trop^{-1}(w)\subset \Hom(M,\Ka^{\times})$.

For every $\Ra$-scheme $\cX$, we denote by $\spe_{\cX}$ the specialization map
 $$\spe_{\cX}:\cX(\Ra)\to \cX(\ka)$$
 defined by reducing coordinates modulo the maximal ideal in $\Ra$. For every scheme $\cY$ over $R_0$, resp.~$R$, we will also write $\spe_{\cY}$ instead of
 $\spe_{\cY\times_{R_0}\Ra}$, resp.~$\spe_{\cY\times_{R}\Ra}$. If $C$ is a constructible subset of $\cY_k$, then we will usually write $\spe^{-1}_{\cY}(C)$ instead of
 $\spe^{-1}_{\cY}(C(\ka))$ to simplify the notation.

 If $X$ is a $K$-scheme of finite type, then an $R$-model of $X$ is a flat $R$-scheme
 of finite type $\cX$ endowed with an isomorphism $\cX_K\to X$.
 By a variety over a field, we mean a scheme of finite type.

\subsection*{Acknowledgements} We are grateful to Ben Davison, Antoine Ducros, Arthur Forey, Davesh Maulik and Vivek Shende for inspiring discussions. Davesh Maulik has informed us that he has another proof of the integral identity of Kontsevich and Soibelman, based on a residue formula for motivic vanishing cycles.
Maulik's paper was still in preparation at the moment of writing.
  Johannes Nicaise is supported by the ERC Starting Grant MOTZETA (project 306610) of the European Research Council, and  by long term structural funding (Methusalem
grant) of the Flemish Government.  Sam Payne is supported in part by NSF CAREER DMS--1149054.

\section{Motivic volumes with Galois action}\label{sec:motvol}
\subsection{Good Galois actions on schemes}
 Let $\cX$ be an $R$-scheme of finite type equipped with a left action of $\gal$ such that
 the morphism $\cX\to\Spec R$ is equivariant. We say that the $\gal$-action on $\cX$ is {\em good} if we can cover
 $\cX$ with $\gal$-stable affine open subschemes $\cU$ such that
 $\gal$ acts continuously on $\mathcal{O}(\cU)$, where we consider the profinite topology on $\gal$ and the discrete topology on $\mathcal{O}(\cU)$.
 The continuity of this action is equivalent to the property that the action on each element of $\mathcal{O}(\cU)$ factors through $\mu_n(k)$ for some $n>0$.

  If the $\gal$-action on $\cX$ is good, then by Galois descent, $X_0=\cX_K/\gal$ is a variety over $K_0$ and the natural map of $K$-varieties
  $$\cX_K\to X_0\times_{K_0}K$$ is a $\gal$-equivariant isomorphism.
 Conversely, let $X_0$ be a variety over $K_0$ and set $X=X_0\times_{K_0}K$, endowed with the Galois action of $\gal$.
 Let $\cX$ be an $R$-model of $X$ such that the $\gal$-action on $X$ extends to $\cX$, and assume that we can cover $\cX$ with $\gal$-stable affine open subschemes.
 Then the $\gal$-action on $\cX$ is good.

If the structure map $\cX\to \Spec R$ factors through $\Spec k$, then the $\gal$-action on $\cX$ is good if and only
if it factors through a finite quotient $\mu_n(k)$ and we can cover $\cX$ with $\gal$-stable affine open subschemes.
Thus, in this case, our definition is equivalent to the one in \cite[\S2.4]{DL}.

\subsection{Polyhedra and constructible sets}
Let $V$ be a finite dimensional affine space over $\QQ$, that is, a torsor under a finite dimensional $\QQ$-vector space.
A {\em polyhedron} in $V$ is a finite intersection of closed rational half-spaces in $V$. In other words, it is a set of the form
$$\{v\in V\,|\,f_i(v)\geq 0\mbox{ for }i=1,\ldots,r\}$$
where $f_1,\ldots,f_r$ are affine linear maps from $V$ to $\QQ$.
 A {\em constructible} subset $\Gamma$ of $V$ is a finite Boolean combination of polyhedra.

 There exists a unique $\ZZ$-valued invariant $\chi'$ on the Boolean algebra of constructible subsets in $V$ that is additive on disjoint unions
 and that assigns the value $1$ to every non-empty polyhedron. This invariant $\chi'$ is called the {\em bounded Euler characteristic}.
  It is clear from the definition that it is invariant
 under affine
 linear automorphisms of $V$.
 One can compute $\chi'(\Gamma)$ for
 every constructible subset $\Gamma$ of $V$ in the following way. We choose an isomorphism of affine spaces $V\to \QQ^n$ for some $n\geq 0$.
  There is a canonical subset $\Gamma_{\R}$ of $\R^n$ associated with $\Gamma$, defined by the same system of $\Q$-linear inequalities as $\Gamma$.
  Then one can show that the compactly supported Euler characteristic of
 $\Gamma_{\R}\cap [-r,r]^n$ stabilizes for sufficiently large $r\in \RR$; the limit value is precisely $\chi'(\Gamma)$.

Constructible sets with vanishing bounded Euler characteristic will play an important role in the applications in Sections~\ref{sec:appli} and \ref{sec:general}.  Typical examples include half-open line segments and open half-lines, as well as products of these with arbitrary constructible sets.

\subsection{Semi-algebraic sets}\label{ss:semialg}
 Let $X$ be a variety over $\Ka$.
 A semi-algebraic subset of $X$ is a finite Boolean combination of subsets of $X(\Ka)$ of the form
  \begin{equation}\label{eq:sa}
  \{x\in U(\Ka)\,|\,\val(f(x))\leq \val(g(x))\} \subset X(\Ka)
  \end{equation} where $U$ is an affine open subvariety of $X$ and $f,\,g$ are regular functions on $U$.
 If $X$ is of the form $X_0\times_{\KK_0}\Ka$, for some variety $X_0$ over $K_0$, then we say that $S$ is defined over $\KK_0$ if we can write it as a finite Boolean combination of sets of the form \eqref{eq:sa} such that
  $U$, $f$ and $g$ are defined over $\KK_0$. Note that this property depends on the choice of $X_0$; if we want to make this choice explicit, we will also call $S$ a semi-algebraic subset of $X_0$
  (even though it is not an actual subset of $X_0$).

  \begin{exam}  Let $T_0$ be a split $\KK_0$-torus with cocharacter lattice $N$.
 Then for every constructible subset $\Gamma$ of $N_{\Q}$, the set $$\trop^{-1}(\Gamma)\subset T_0(\Ka)$$
 is a semi-algebraic subset of $T_0$.
\end{exam}
\begin{exam}
 Let $\cX$ be an $\Ra$-scheme of finite type and let $C$ be a constructible subset of $\cX_{\ka}$. Then $\spe^{-1}_{\cX}(C)$ is a semi-algebraic subset of $\cX_{\Ka}$.
 To see this, it suffices to consider the case where $\cX$ is affine and $C$ is closed in $\cX_{\ka}$.
  If $(z_1,\ldots,z_r)$ is a tuple of generators of the $\Ra$-algebra $\mathcal{O}(\cX)$, and $(f_1,\ldots,f_s)$ is a tuple of elements of $\mathcal{O}(\cX)$ such that
  $C$ is the set of common zeroes of the functions $f_i$, then
  $$\spe_{\cX}^{-1}(C)=\{x\in \cX(\Ka)\,|\,\val(z_i(x))\geq 0,\,\val(f_j(x))>0\mbox{ for all }i,j\}.$$
  This is a finite Boolean combination of sets of the form \eqref{eq:sa}. The same argument shows that, if $\cX_0$ is an $R_0$-scheme of finite type
and $C_0$ is a constructible subset of $(\cX_0)_k$, then $\spe^{-1}_{\cX_0}(C_0)$ is a semi-algebraic subset of $(\cX_0)_{K_0}$.
  \end{exam}

If $X$ and $X'$ are varieties over $\Ka$ and
 $S$ and $S'$ are semi-algebraic subsets of $X$ and $X'$, respectively, then a morphism of semi-algebraic sets $f:S\to S'$ is a map whose graph is semi-algebraic
 in $X\times_{\Ka} X'$. If $X=X_0\times_{\KK_0} \Ka$ and $X'=X'_0\times_{\KK_0}\Ka$ and $S$ and $S'$ are defined over $\KK_0$,
 then we say that $f$ is defined over $\KK_0$ if its graph has this property. It follows from Robinson's quantifier elimination for algebraically closed valued fields
 that the image of a morphism of semi-algebraic sets is again a semi-algebraic set. If the morphism is defined over $\KK_0$, then the same holds for its image.

 We denote by $\VF_{\KK_0}$ the category of semi-algebraic sets defined over $\KK_0$; it comes equipped with a base change functor $\VF_{\KK_0}\to \VF_{\Ka}$ to the category of semi-algebraic sets over $\Ka$.
 For every object $S$ in $\VF_{\KK_0}$, there is a natural action of the Galois group $\gal$ on the set $S$, and $\KK_0$-morphisms of semi-algebraic sets are equivariant with respect to this action.

\begin{prop}\label{prop:galequiv}
Let $X_0$ be a variety over $K_0$, and let $S$ be a semi-algebraic subset of $X=X_0\times_{K_0}\Ka$. Then $S$ is defined over $K_0$ if and only if
$S$ is stable under the Galois action of $G=\Gal(\Ka/\KK_0)$ on $X_0(\Ka)$.
\end{prop}
\begin{proof}
The condition is clearly necessary; we will prove that it is also sufficient.
 Suppose that $S$ is stable under the $G$-action. We may assume that $X_0$ is affine;
 then $S$ is a finite Boolean combination of sets of the form
 $$\{x\in X(\Ka)\,|\,\val(f(x))\leq \val(g(x))\}$$ with $f,g\in \mathcal{O}(X)$.
We can find a finite extension $K'$ of $K_0$ in $\Ka$ such that all the functions $f$ and $g$ that appear in these expressions
are defined over $K'$.
Set $X'=X_0\times_{K_0}K'$. We view $X'$ as a variety over $K_0$
by forgetting the $K'$-structure. Let $S'$ be the semi-algebraic subset of $X'(\Ka)$ defined by the same formulas as $S$; this is a semi-algebraic subset of $X'$ defined over $K_0$.
  Since $S$ is stable under the $G$-action on $X_0(\Ka)$, the image of $S'$ under the projection map $X'(\Ka)\to X_0(\Ka)$ is equal to $S$.
 Now it follows from quantifier elimination that $S$ is defined over $K_0$.
\end{proof}

 \begin{exam}\label{exam:descent}
We will use Proposition \ref{prop:galequiv} in the following way. Let $X_0$ be a variety over $\KK_0$, and let $\cX$ be an $R$-model of $X=X_0\times_{\KK_0}K$ such that the Galois action of $\gal$ on $X$ extends
to an action on $\cX$.
  Let $C$ be a constructible subset of $\cX_k$ that is stable under the action of $\gal$. Then $\spe^{-1}_{\cX}(C)$ is stable under the $G$-action on
   $X_0(\Ka)$. Hence, it is a semi-algebraic subset of $X_0$. If the $\gal$-action on $\cX$ is good, this can also be seen directly:
we can form the quotient $\cX_0=\cX/\gal$ in the category of schemes. This is an $R_0$-scheme of finite type whose generic fiber is canonically isomorphic with $X_0$.
If we denote by $C_0$ the image of $C$ under the projection morphism $\cX\to \cX_0$, then $C_0$ is a constructible subset of $(\cX_0)_k$ and $\spe^{-1}_{\cX}(C)=\spe_{\cX_0}^{-1}(C_0)$.
 \end{exam}

\subsection{Grothendieck rings of varieties and semi-algebraic sets}
The piecewise geometry of varieties and semi-algebraic sets is encoded in various Grothendieck rings.
 We first consider the Grothendieck ring $\gro^{\gal}(\Var_k)$ of $k$-varieties with $\gal$-action. As an abelian group, it is defined by the following presentation:
\begin{itemize}
\item {\em Generators}: isomorphism classes of $k$-varieties $X$ endowed with a good $\gal$-action.
 Isomorphism classes are taken with respect to $\gal$-equivariant isomorphisms.
\item {\em Relations}: we consider two types of relations.
\begin{enumerate}
\item {\em Scissor relations}: if $X$ is a $k$-variety with a good $\gal$-action and $Y$ is a $\gal$-stable closed subvariety of $X$, then
$$[X]=[Y]+[X\setminus Y].$$
\item {\em Trivialization of linear actions}: let $X$ be a $k$-variety with a good $\gal$-action, and let $V$ be a $k$-vector space
of dimension $d$ with a good linear action of $\gal$.
Then $$[X\times_k V]=[X\times_k \AA^d_k]$$ where the $\gal$-action on $X\times_k V$ is the diagonal action and the action on $\AA^d_k$ is trivial.
\end{enumerate}
\end{itemize}
The group $\gro^{\gal}(\Var_k)$ has a unique ring structure such that $[X]\cdot [X']=[X\times_k X']$ for all $k$-varieties $X$, $X'$ with good $\gal$-action. Here the $\gal$-action
on $X\times_k X'$ is the diagonal action.
 The identity element in $\gro^{\gal}(\Var_k)$ is the class of the point $\Spec k$.
 We write $\LL$ for the class of $\AA^1_k$ (with the trivial $\gal$-action) in the ring $\gro^{\gal}(\Var_k)$.

\begin{rem}
The trivialization of linear actions is a standard operation in the theory of motivic integration, in order to obtain well-defined motivic measures and a change of variables
formula; see for instance \cite[\S2.4]{DL}. All of the standard cohomological realizations respect this trivialization. In the context of Hrushovski and Kazhdan's theory of motivic integration,
 this relation naturally appears when one identifies all of the fibers of the tropicalization map
 $$\trop:(K^{\times})^n\to \Q^n$$ in the definition of the motivic volume (see Theorem~\ref{thm:HK}).
\end{rem}

Now, we define the Grothendieck ring $\gro(\VF_{\KK_0})$ of semi-algebraic sets over $\KK_0$.
The underlying group is the free abelian group on isomorphism classes $[S]$ of semi-algebraic sets $S$ over $\KK_0$
modulo the relations $$[S]=[S']+[S\smallsetminus S']$$ for all semi-algebraic sets $S'\subset S$. Here isomorphism classes are taken in the category $\VF_{\KK_0}$, that is, with respect to semi-algebraic bijections
 defined over $\KK_0$.
The group $\gro(\VF_{\KK_0})$ has a unique ring structure such that $$[S]\cdot [S']=[S\times S']$$ in $\gro(\VF_{\KK_0})$ for all semi-algebraic sets $S$ and $S'$.
 The identity element in $\gro(\VF_{\KK_0})$ is the class of the point, that is, the semi-algebraic set $[X_0(\Ka)]$ with $X_0=\Spec \KK_0$.

\begin{exam}\label{exam:ball}
Let
$$B=\{x\in \Ka\,|\,\val(x)>0\}$$ be the open unit ball in $\Ka$. This is a semi-algebraic set defined over $\KK_0$. We can write the class of $B$ in $\gro(\VF_{\KK_0})$ as
$$[B]=1 + [\trop^{-1}(\Q_{>0})]$$ (the class of the point $\{0\}$ plus the class of the punctured ball).
\end{exam}

\subsection{A refinement of the motivic volume}
In \cite{HK}, Hrushovski and Kazhdan have defined the {\em motivic volume} of a semi-algebraic set over $\Ka$.
 More precisely, they constructed a ring morphism
 $$\Vol:\gro(\VF_{\Ka})\to \gro(\Var_{\ka})$$ from the Grothendieck ring of semi-algebraic sets over $\Ka$ to the Grothendieck ring of varieties over $\ka$.
 Their construction is based in an essential way on the model theory of algebraically closed valued fields.
  The main lines are presented from a geometric perspective in \cite{NPS} and \cite{Ni-HK}.
  If the semi-algebraic set is defined over the subfield $\KK_0$ of $\Ka$, then the motivic volume can be refined
  in order to reflect the action of the Galois group $\Gal(\Ka/\KK_0)$. We will now explain this refinement, again presenting
  the results of Hrushovski and Kazhdan in a more geometric language.

\begin{thm}[Hrushovski-Kazhdan]\label{thm:HK}
There exists a unique ring morphism
$$\Vol:\gro(\VF_{\KK_0})\to \gro^{\gal}(\Var_k)$$ that satisfies the following properties.
\begin{enumerate}
\item \label{it:smoothvol} Let $X_0$ be a smooth variety over $\KK_0$, and let $\cX$ be a smooth $R$-model of $X=X_0\times_{\KK_0}\KK$ such that the Galois action of $\gal$ on $X$ extends to a good action on $\cX$.
 Then $S=\cX(\Ra)$ is defined over $\KK_0$, and $\Vol([S])=[\cX_k]$ in $\gro^{\gal}(\Var_k)$.

\item Let $\Gamma$ be a constructible subset of $\Q^n$, for some $n\geq 0$, and set $S'=\trop^{-1}(\Gamma)$. Then $S'$ is defined over $\KK_0$, and
$$\Vol([S'])=\chi'(\Gamma)(\LL-1)^n$$ in $\gro^{\gal}(\Var_k)$.
\end{enumerate}
\end{thm}
\begin{proof}
We unravel some of the central results in \cite{HK}.
 We will not explain all of the notations, as this is not strictly necessary to follow the argument, but we provide precise references for the reader.

 Hrushovski and Kazhdan constructed a surjective morphism of rings
$$\Theta:\gro(\RES[\ast])\otimes_{\ZZ}\gro(\Q[\ast])\to \gro(\VF_{\KK_0})$$
and gave an explicit description of its kernel -- see Theorem~8.8 and Corollary~10.3 of \cite{HK}.
 Here $\gro(\RES[\ast])$ and $\gro(\Q[\ast])$ are certain graded Grothendieck rings of varieties with $\gal$-action and constructible sets in $\Q$-affine spaces, respectively.
  Informally speaking, the relations that generate the kernel express that the fibers of the tropicalization map are $\GG^n_{m,R}(R)$-torsors,
 and that the open unit ball in $K$ from Example \ref{exam:ball} can also be described as $\spe^{-1}_{\AA^1_{R_0}}(0)$.
 In Theorem~10.5(4) of \cite{HK} and its proof, Hrushovski and Kazhdan have also defined a ring morphism
 $$ \mathcal{E}':\gro(\VF_{\KK_0})\to \gro^{\gal}(\Var_k).$$
 To be precise, in \cite{HK} the target of $\mathcal{E}'$ is a quotient of $\gro(\RES[\ast])$ they denote by $!\gro(\RES)$, but this ring is canonically isomorphic to
 $\gro^{\gal}(\Var_k)$ by \cite[4.3.1]{HL} (there it was assumed that $k$ is algebraically closed, but the proof remains valid if we only assume that $k$ contains all the roots of unity). We set $\Vol=\mathcal{E}'$ and we will prove that it satisfies, and is uniquely determined by, the properties in the statement.

  The ring $\gro(\Q[\ast])$ is generated by the classes of pairs $(\Gamma,n)$ where $n$ is a nonnegative integer and $\Gamma$ is
  a constructible subset of $\Q^n$. The image of the class of $(\Gamma,n)$ under $\mathcal{E}'\circ \Theta$ is precisely
  $\chi'(\Gamma)(\LL-1)^n$.
 Thus it suffices to prove the following two claims:
 \begin{enumerate}[(a)]
 \item \label{it:claim1} The ring $\gro(\VF_{\KK_0})$ is generated by elements of the form $[S\times \trop^{-1}(\Gamma)]$ where $S$ is as in the statement of Theorem~\ref{thm:HK} and $\Gamma$ is a polyhedron in $\Q^n$ for some $n\geq 0$.
 \item \label{it:claim2} The morphism $\mathcal{E}'$ sends the class of $S$ in $\gro(\VF_{\KK_0})$ to the class of $\cX_k$ in $\gro^{\gal}(\Var_k)$.
 \end{enumerate}
 These statements imply the existence and uniqueness of the morphism $\Vol$.

  Let us prove these claims.
  The ring $\gro(\RES[\ast])$ is generated by equivalence classes of pairs $(Y,n)$ where $n$ is a nonnegative integer and $Y$ is a $k$-variety of
  pure dimension $d\leq n$ endowed with a good action of $\widehat{\mu}$ (here we are implicitly using the identifications explained in \cite[\S4.3]{HL}).
     Partitioning $Y$ into subvarieties, we may assume that $Y$ is smooth.
  We set $$\cY=Y\times_k R$$ and we endow it with the diagonal $\widehat{\mu}$-action.
  We write $B$ for the open unit ball in $\Ka$ as in Example \ref{exam:ball}.
   It follows easily from the constructions in \cite{HK} that the image of the class of $(Y,n)$ under $\Theta$ is $$[\cY(\Ra)][B^{n-d}]\in \gro(\VF_{\KK_0}).$$
    This proves claim \eqref{it:claim1}.

 It remains to prove claim \eqref{it:claim2}. Let $\cX$ and $S$ be as in the statement of Theorem~\ref{thm:HK}.
 We set $$\cX'=\cX_k\times_k R$$ and we endow it with the diagonal $\widehat{\mu}$-action.
 By definition, the image of $[\cX(\Ra)]$ under $\mathcal{E}'$ is $$[\cX_k]\in \gro^{\gal}(\Var_k).$$
 Thus it suffices to show that there exists a semi-algebraic bijection between $\cX(\Ra)$ and $\cX'(\Ra)$ that is defined over $\KK_0$.
 Working locally on $\cX$ and $\cX'$, we may assume that there exist \'etale $R$-morphisms $\cX\to \AA^n_R$ and $\cX'\to \AA^n_R$ that coincide on the special fibers (we are not requiring any $\gal$-equivariance here). Set $\cZ=\cX\times_{\AA^n_R}\cX'$ and denote by $\Delta$ the image of the diagonal morphism
 $\cX_k\to \cZ_k$. Using the henselian property for $R$, we see that $S'=\spe^{-1}_{\cZ}(\Delta)$ is the graph of a bijection between $\cX(\Ra)$ and $\cX'(\Ra)$.
  We will prove that  $S'$ is defined over $\KK_0$.

 The set $S'$ is stable under the Galois action of $\Gal(\ka/k)$ by construction. Thus by Proposition~\ref{prop:galequiv},
  it is enough to prove that $S'$ is stable under the action of $\gal$ on $\cX(\Ra)\times \cX'(\Ra)$, that is, the bijection defined by $S'$ is $\gal$-equivariant.
  Denote by $\fX$ and $\fX'$ the formal $t$-adic completions  of $\cX$ and $\cX'$, and denote by $\fZ$ the open formal subscheme of the formal $t$-adic completion of $\cZ$ supported on the open subscheme $\Delta$ of $\cZ_k$.
   Then $\fZ$ is the graph of an isomorphism of formal $\widehat{R}$-schemes $h:\fX\to \fX'$. The induced isomorphism $h_k$ between the special fibers is $\gal$-equivariant
   by construction. This implies that $h$ is $\gal$-equivariant, because every continuous action of $\gal$ on
    $\fX$ or $\fX'$ by $\widehat{R}$-automorphisms that are trivial on the special fiber, is trivial (to see this, linearize the action on the completed local rings).
  \end{proof}

If $S$ is a semi-algebraic set defined over $\KK_0$, we will write $\Vol(S)$ for $\Vol([S])$.
 If $X$ is a variety over $\KK_0$, then we can view $X(\Ka)$ as a semi-algebraic set defined over $\KK_0$; we will usually write $\Vol(X)$ instead of $\Vol(X(\Ka))$.
 It follows immediately from the definitions that the motivic volume has the following properties with respect to extensions of the base field $K_0$.
 Let $K'_0$ be a finite extension of $K_0$ in $\Ka$. Denote by $k'$ the residue field of $K'_0$ and let $\gal'$ be the inertia group of $K'_0$; this is an open subgroup of $\gal$.
Let
$$\mathrm{Res}^{\gal}_{\gal'}:\gro^{\gal}(\Var_k)\to \gro^{\gal'}(\Var_{k'})$$ be the morphism defined by base change to $k'$ and restricting the $\gal$-action to $\gal'$. Then
 the diagram
$$\xymatrix{
\gro(\VF_{\KK_0}) \ar[d] \ar[r]^{\Vol} &\gro^{\gal}(\Var_k)\ar[d]^{\mathrm{Res}^{\gal}_{\gal'}} \\
\gro(\VF_{\KK'_0}) \ar[r]_{\Vol} &\gro^{\gal'}(\Var_{k'})} $$
 commutes, where the left vertical morphism is the base change morphism.
Likewise, the diagram
$$\xymatrix{
\gro(\VF_{\KK_0}) \ar[d] \ar[r]^{\Vol} &\gro^{\gal}(\Var_k)\ar[d]^{\mathrm{Res}^{\gal}_{\{1\}}} \\
\gro(\VF_{\Ka}) \ar[r]_{\Vol} &\gro(\Var_{\ka})} $$
 commutes, where the left vertical morphism is the base change morphism and $\mathrm{Res}^{\gal}_{\{1\}}$ is the morphism that performs base change to $\ka$ and forgets the $\gal$-action.

\begin{exam}\label{exam:ballvol}
If $B$ is the open unit ball in $\Ka$ from Example \ref{exam:ball}, then
$\Vol(B)=1$ in $\gro^{\gal}(\Var_k)$ because $\chi'(\Q_{>0})=0$.
\end{exam}

\if false
 The formula for the motivic volume in Theorem \ref{thm:HK}\eqref{it:smoothvol} admits the following useful generalization.

\begin{prop}\label{prop:goodredvol}
Let $X_0$ be a smooth variety over $\KK_0$, and let $\cX$ be a smooth $R$-model of $X=X_0\times_{\KK_0}\KK$ such that the Galois action of $\gal$ on $X$ extends to a good action on $\cX$. Let $C$ be a subscheme of $\cX_k$ that is stable under the action of $\gal$.
 Then $S=\spe^{-1}_{\cX}(C)$ is defined over $\KK_0$, and $\Vol(S)=[C]$ in $\gro^{\gal}(\Var_k)$.
\end{prop}
\begin{proof}
 The semi-algebraic set $S$ is defined over $\KK_0$ by Proposition \ref{prop:galequiv}.
To establish Claim \ref{it:claim2} in the proof of Theorem \ref{thm:HK}, we have constructed a $\gal$-equivariant semi-algebraic bijection between $\cX(\Ra)$
 and $(\cX_k \times_k R)(\Ra)$. Since this bijection commutes with the respective specialization maps, we may assume that $\cX=\cX_k\times_k R$ (endowed with the diagonal $\gal$-action). By the additivity of the motivic volume and Noetherian induction, is suffices to show that the proposition holds if we replace $C$ by some $\gal$-stable non-empty open subscheme. Thus we may assume that $C$ is smooth over $k$ and that the $\gal$-action on $C$ factors through a free action of $\mu_n(k)$ for some $n>0$.
  Let $\gal'$ be the kernel of the projection $\gal\to \mu_n(k)$. Generically, $C$ is defined in $\cX_k$ by an ideal $I$ generated by a regular sequence $(f_1,\ldots,f_s)$
  $\cX_k\to \AA^r_k$ such that $C$ is the inverse image of the linear subspace $\AA^s_k$, for some $r\geq s\geq 0$.  Now both $\cX$ and $C\times_k \AA^{s}_{R}$ are $\gal$-equivariant \'etale covers of
  $\AA^r_{R}\times_R \AA^{s}_R$. Denote by $\cZ$ their fiber product over $\AA^r_{R}\times_R \AA^s_{R}$  and by $\Delta$ the image of $C$ in $\cZ_k$ under the diagonal embedding.
   We denote by $B$ the open unit ball in $\Ka$; it has motivic volume $1$ by Example \ref{exam:ballvol}.
  Then $\spe^{-1}_{\cZ}(\Delta)$ is the graph of a semi-algebraic bijection  $$\spe^{-1}_{\cX}(C)\to
  \spe^{-1}_{C\times_k \AA^{r-s}_{R}}(C)=C(\Ra)\times B^{r-s}$$ that is defined over $\KK_0$ because it is $\gal$-equivariant.
   The equality $\Vol(S)=[C]$ in $\gro^{\gal}(\Var_k)$ now follows from Theorem \ref{thm:HK}.
\end{proof}
\fi

\subsection{Comparison with the motivic nearby fiber of Denef and Loeser}\label{sec:DL}
 Let $\cX$ be a $R_0$-scheme of finite type. Assume that $\cX$ is regular and that its special fiber $\cX_k$ is a divisor with strict normal crossings support.
  In this section, we will establish an explicit formula for
  the motivic volume $\Vol(\cX(\Ra))\in \gro^{\gal}(\Var_k)$ of the semi-algebraic set $\cX(\Ra)$. This formula will then allow us to compare the motivic volume with the motivic nearby fiber of Denef and Loeser.

 We write $$\cX_k=\sum_{i\in I}N_i E_i$$ where $E_i,i\in I$ are the irreducible components of $\cX_k$ and the coefficients $N_i$ are their multiplicities in $\cX_k$.
 For every non-empty subset $J$ of $I$, we set
 $$E_J=\bigcap_{j\in J}E_j,\quad E_J^o=E_J\setminus \left(\bigcup_{i\notin J}E_i\right).$$
 The sets $E^o_J$ and $E_J$ are locally closed subsets of $\cX_k$, and we endow them with their induced reduced subscheme structure.
  By the definition of a strict normal crossings divisor, all of the schemes $E_J$ and $E_J^o$ are regular.
 As $J$ ranges through the non-empty subsets of $I$, the subschemes $E_J^o$ form a partition of $\cX_k$.

 Set $e=\mathrm{lcm}\{N_i\,|\,i\in I\}$.
 Let $\widetilde{\cX}$ be the normalization of $\cX\times_{k\llb t\rrb}k\llb t^{1/e}\rrb$ and set
 $$\widetilde{E}_J^o=(\widetilde{\cX}\times_{\cX}E_J^o)_{\red}$$ for every non-empty subset $J$ of $I$.
  Then the group $\mu_e(k)$ acts on $\widetilde{E}_J^o$, and this action factors through a free action of $\mu_{N_J}(k)$ where
  $N_J=\gcd\{N_j\,|\,j\in J\}$. This makes $\widetilde{E}_J^o$ into a $\mu_{N_J}(k)$-torsor over $E_J^o$  -- see \cite[\S2.3]{Ni-tame} and \cite[3.2.2]{BuNi}.
  We denote by $h:\widetilde{\cX}\to \cX$ the projection morphism from $\widetilde{\cX}$ to $\cX$.

  \begin{thm}\label{thm:snc}
  Let $C$ be a locally closed subset of $\cX_k$ and set $S=\spe^{-1}_{\cX}(C)$.
  Then we have
  $$\Vol(S)=\sum_{\emptyset \neq J \subset I}(1-\LL)^{|J|-1}[\widetilde{E}_J^o\cap h^{-1}(C)]$$ in $\gro^{\gal}(\Var_k)$.
  \end{thm}
  \begin{proof}
    Our proof follows similar lines as that of claim \eqref{it:claim2} in the proof of Theorem~\ref{thm:HK}, but we need to replace the \'etale-local
    model $\AA^n_R$ to take the singularities of $\cX_k$ into account.
        By additivity, we may assume that $C$ is a closed subset of the stratum $E_J^o$ for some non-empty subset $J$ of $I$.
        We set $M_j=N_j/N_J$ for every $j\in J$. Set $R'=k\llb t^{1/N_J} \rrb$ and denote by $\cY$ the normalization of $\cX\times_{R_0}R'$.
         The scheme $\cY$ carries a natural $\mu_{N_J}(k)$-action that is compatible with the Galois action on $R'$.
        It follows from \cite[3.2.2]{BuNi} that the natural morphism $\widetilde{\cX}\to \cY$ induces a $\mu_{N_J}(k)$-equivariant isomorphism
     of $E^o_J$-schemes   $$\widetilde{E}_J^o\to \cY\times_{\cX} E^o_J.$$
     We write $\widetilde{C}$ for the inverse image of $C$ in $\cY\times_{\cX} E^o_J$. Under the above isomorphism, it corresponds to
     the closed subset $h^{-1}(C)$ of $\widetilde{E}_J^o$.

 Working locally on $\cX$, we may assume that $\cX$ is affine, that $I=J$, and that $E_j$ is defined by a global equation $f_j=0$ on $\cX$ for every $j\in J$.
  Then we can write $$t=u\prod_{j\in J}f^{N_j}_j$$ with $u$ an invertible function on $\cX$. The proof of \cite[2.3.2]{Ni-tame} reveals that
  there exists a $\mu_{N_J}(k)$-equivariant isomorphism of $\cX$-schemes
  $$\cY\to \Spec \mathcal{O}(\cX)[T]/(1-uT^{N_J})$$ where $\mu_{N_J}(k)$ acts on the target by multiplication on $T$ and such that
   $t^{1/N_J}\in \mathcal{O}(\cY)$ is identified with $T^{-1}\prod_{j\in J}f_j^{M_j}$.
   We choose integers $a_j,\,j\in J$ such that $\sum_{j\in J}a_jM_j=1$, and we set $g_j=T^{-a_j}f_j$ for every $j\in J$.
  Then
  $$t^{1/N_J}=\prod_{j\in J}g_j^{M_j} $$
and the functions $g_j$ give rise to a smooth morphism
 $$g:\cY\to \cU=\Spec R'[u_j,j\in J]/(t^{1/N_J}-\prod_{j\in J}u_j^{M_j})$$ such that $\widetilde{E}_J^o$ is the inverse image of the origin $O$ in $\cU_k$.
  We endow $\cU$ with the left $\mu_{N_J}(k)$-action induced by the Galois action on $R'$ and the rule that $ u_j\ast \zeta=\zeta^{-a_j}u_j$ for every $\zeta\in \mu_{N_J}(k)$.
   Then the morphism $g$ is $\mu_{N_J}(k)$-equivariant.
 Shrinking $\cX$ if necessary, we can further arrange that $g$ lifts to an \'etale morphism
 $$\cY\to \cU\times_{R'}\AA^m_{R'}$$ for some $m\geq 0$ (which is not necessarily $\mu_{N_J}$-equivariant).
 Corestricting this morphism over $O\times_k \AA^m_k$ we obtain
an \'etale morphism $\widetilde{E}_J^o\to \AA^m_k$.
 If we set $\cY'=\cU\times_k \widetilde{E}_J^o$, then
this morphism gives rise, at its turn, to an \'etale morphism
  $$\cY'\to \cU\times_{R'} \AA^m_{R'}.$$
We endow $\cY'$ with the diagonal $\mu_{N_J}(k)$-action.

 We consider the fibered product $$\cZ=\cY\times_{(\cU\times_{R'} \AA^m_{R'})} \cY'.$$
  Denote by $\Delta$ the image of $\widetilde{C}$ under the diagonal map $\widetilde{E}_J^o\to \cZ_k$.
  The henselian property for $R$ implies that $\spe^{-1}_{\cZ}(\Delta)$ is the graph of a semi-algebraic bijection $\alpha$ between
  $\spe^{-1}_{\cY}(\widetilde{C})$ and $$\spe^{-1}_{\cY'}(\widetilde{C})=\spe^{-1}_{\cU}(O)\times \spe^{-1}_{\widetilde{E}_J^o\times_k R'}(\widetilde{C}).$$
     Now we make the following claims.
  \begin{enumerate}[(a)]
  \item \label{it:claim1b} The bijection $\alpha$ is defined over $\KK_0$.
    \item \label{it:claim2b} The volume of $\spe^{-1}_{\widetilde{E}_J^o\times_k R'}(\widetilde{C})$ is equal to $[\widetilde{C}]$ in $\gro^{\gal}(\Var_k)$.
  \item \label{it:claim3b} The volume of $\spe^{-1}_{\cU}(O)$ is equal to $(1-\LL)^{|J|-1}$.
  \end{enumerate}
  These claims together yield the desired formula for $\Vol(\spe^{-1}_{\cX}(C))$.

 In order to prove claim \eqref{it:claim1b} it suffices to show that $\alpha$ is $\gal$-equivariant, by Proposition~\ref{prop:galequiv} (applied to the graph of $\alpha$).
 We denote by $\fY$ and $\fY'$ the formal completions of $\cY$ and $\cY'$ along $\widetilde{C}$, respectively.
 Then, once again, the formal completion of $\cZ$ along $\Delta$ is the graph of an isomorphism
  $\alpha':\fY\to \fY'$. The bijection $\alpha$ is map induced by $\alpha'$
  on $\Ra$-points.
   By construction, both $\fY$ and $\fY'$ come equipped with a smooth $\mu_{N_J}(k)$-equivariant morphism to $\fU$, the formal completion of $\cU$ at $O$,
   and $\alpha'$ is an isomorphism of formal $\fU$-schemes.
     Now the result follows from the fact that every finite order automorphism of $\fY$ or $\fY'$ over $\fU$
     that acts trivially on the fiber over $O$ is the identity (this can again be seen
     by linearizing the action on the completed local rings).

  Next, we prove claim \eqref{it:claim2b}. By additivity, we may assume that $C$ is a smooth closed subvariety of $E_J^o$.
  Then, locally on $E_J^o$, we can find an \'etale morphism
  $E_J^o\to \AA^r_k$ such that $C$ is the inverse image of the linear subspace $\AA^s_k$, for some $r\geq s\geq 0$.
  Now both $\widetilde{E}_J^o\times_k R'$ and $\widetilde{C}\times_k \AA^{r-s}_{R'}$ are equivariant \'etale covers of
  $\AA^r_{R'}$. Denote by $\cV$ their fiber product over $\AA^r_{R'}$  and by $\Delta'$ the image of $\widetilde{C}$ in $\cV_k$ under the diagonal embedding.
   We denote by $B$ the open unit ball in $\Ka$; it has motivic volume $1$ by Example \ref{exam:ballvol}.
  Then $\spe^{-1}_{\cV}(\Delta')$ is the graph of a semi-algebraic bijection  $$\spe^{-1}_{\widetilde{E}_J^o\times_k R'}(\widetilde{C})\to
  \spe^{-1}_{\widetilde{C}\times_k \AA^{r-s}_{R'}}(\widetilde{C})=\widetilde{C}(\Ra)\times B^{r-s}$$ that is defined over $\KK_0$ because it is $\gal$-equivariant.
  Claim \eqref{it:claim2b} now follows from the definition of the motivic volume.

  Finally, we prove claim \eqref{it:claim3b}. We denote by $K'$ the fraction field of $R'$ and we set $r=|J|$ and $v_1=(M_j,\,j\in J)$.
  Since the entries $M_j$ are coprime, $v_1$ can be extended to a basis $v_1,\ldots,v_r$ of $\ZZ^J$. For every $i$ in $\{2,\ldots,r\}$ we set
  $b_i=\sum_{j\in J}a_jv_{i,j}$.
  The $(r-1)$-tuple of $\mu_{N_J}(k)$-invariant invertible functions
  $$(t^{-b_2/N_J}\prod_{j\in J}u_j^{v_{2,j}},\ldots,t^{-b_r/N_J}\prod_{j\in J}u_j^{v_{r,j}})$$ defines an isomorphism $\cU_{K'}\to \mathbb{G}^{r-1}_{m,K'}$ that descends to $\KK_0$.
  This isomorphism identifies $\spe^{-1}_{\cU}(O)$ with
  $$\trop^{-1}(\Gamma)\subset (\Ka^{\times})^n$$ where $\Gamma$ is an open $(r-1)$-simplex in $\Q^{r-1}$.
  Since the bounded Euler characteristic of an open $(r-1)$-simplex is equal to $(-1)^{r-1}$,
  the definition of the motivic volume now implies that $\spe^{-1}_{\cU}(O)=(1-\LL)^{|J|-1}$.
  \end{proof}

Using Theorem~\ref{thm:snc}, one can compare the motivic volume to other motivic invariants that appear in the literature. We are mainly interested in
 the motivic nearby fiber of Denef-Loeser \cite[3.5.3]{DL}.
  The motivic nearby fiber was defined as a motivic incarnation of the complex of nearby cycles associated with a morphism of $k$-varieties
  $f:U\to \AA^1_k$ with smooth generic fiber. It is an object $\psi^{\mathrm{mot}}_f$ that lies in the localized Grothendieck ring of varieties with $\gal$-action over the zero
  locus $f^{-1}(0)$ of $f$.
  For every subvariety $C$ of $f^{-1}(0)$ we can restrict $\psi^{\mathrm{mot}}_f$ over $C$ by base change and then view the result as an element in
  $\mathcal{M}^{\gal}_k=\gro^{\gal}(\Var_k)[\LL^{-1}]$ by forgetting the $C$-structure.
     In particular, if $x$ is a closed point on $U$ such that $f(x)=0$, then by restricting $\psi^{\mathrm{mot}}_f$ over $x$ we obtain an element in
  $\mathcal{M}^{\widehat{\mu}}_k$ that is called the {\em motivic Milnor fiber} of $f$ at $x$ and denoted by $\psi^{\mathrm{mot}}_{f,x}$.

  \begin{cor}\label{cor:compar}
 Let $f:U\to \Spec k[t]$ be a morphism of varieties over $k$, with smooth generic fiber,
and denote by $\cX$ the base change of $U$ from $k[t]$ to $R_0=k\llb t\rrb$.
Let $C$ be a subvariety of the zero locus $f^{-1}(0)=\cX_k$ of $f$.
 Then $\spe^{-1}_{\cX}(C)$ is a semi-algebraic set defined over $\KK_0$.
 It consists of the points $u$ in $U(\Ra)$ that satisfy $f(u)=t$ and such that the reduction of $u$ modulo the maximal ideal
 in $\Ra$ belongs to $C$.
 The image of $\Vol(\spe^{-1}_{\cX}(C))$ in the localized Grothendieck ring $\mathcal{M}^{\widehat{\mu}}_k$ is
equal to the restriction of Denef and Loeser's motivic nearby fiber of $f$ over $C$. In particular, $\psi_f^{\mathrm{mot}}=\Vol(\cX(\Ra))$ and,
for every closed point $x$ on $\cX_k$,
 $\psi^{\mathrm{mot}}_{f,x}=\Vol(\spe^{-1}_{\cX}(x))$ in $\mathcal{M}^{\gal}_{k}$.
\end{cor}

\begin{proof}
The motivic nearby fiber can be computed on a log resolution for the pair $(U,f^{-1}(0))$ by means of Denef and Loeser's formula in \cite[3.5.3]{DL}.
The desired equalities then follow immediately from a comparison
with the formula in Theorem~\ref{thm:snc}.
\end{proof}

\begin{rem}\label{rem:compar}  Corollary~\ref{cor:compar} implies, in particular, that the motivic nearby fiber and the motivic Milnor fiber are well-defined
already without inverting $\LL$. This can also be proven directly: one can take Denef and Loeser's formula in terms of a log resolution
 as a definition and use weak factorization to check that it does not depend on the choice of the log resolution. Corollary~\ref{cor:compar} also provides a natural extension
 of the definitions of the motivic nearby fiber and motivic Milnor fiber to the case where the generic fiber of $f$ is singular. This extension coincides with the constructions of Bittner \cite{bittner} and
 Guibert-Loeser-Merle \cite{GLM} after inverting $\LL$. We will discuss a refinement of Corollary~\ref{cor:compar} to an equality in the relative Grothendieck ring over $f^{-1}(0)$ in Corollary~\ref{cor:DLrel}.

  Corollary~\ref{cor:compar} is
 closely related to similar comparison results by Hrushovski and Loeser for the motivic zeta function \cite{HL},
 but our approach is more direct if one only wants to retrieve the motivic nearby fiber; in particular, we do not need to consider the more complicated measured
 version of Hrushosvki and Kazhdan's motivic integration theory, and we avoid inverting $\LL$.
 \end{rem}

\section{A motivic Fubini theorem for the tropicalization map}\label{sec:fubini}
The aim of this section is to develop a flexible tool to compute the motivic volume for a large and interesting class of examples.
 The basic idea is to calculate the volume of a semi-algebraic subset $S$ of an algebraic torus $\mathbb{G}^n_{m,\KK_0}$ by first integrating
 over the fibers of the tropicalization map $\trop:(\Ka^{\times})^n\to \Q^n$ and then integrating the resulting function on $\Q^n$ with respect to the bounded Euler characteristic $\chi'$.
 As mentioned in the introduction, an important advantage of this approach is that $\chi'$ vanishes on bounded half-open intervals and half-bounded open intervals, which allows us in certain cases to discard
 the contribution of pieces of $S$ that are difficult to control directly. Concrete applications will be discussed in Section \ref{sec:appli}.

\subsection{The calculus of constructible functions}

\begin{definition}
Let $A$ be an abelian group and let $V$ be a finite dimensional affine space over $\QQ$. We say that a function $$\varphi : V \rightarrow A$$ is \emph{constructible}
if there exists a partition of $V$ into finitely many
 constructible subsets $\sigma_1,\ldots,\sigma_r$ such that $\varphi$ takes a constant value $a_i\in A$ on $\sigma_i$ for each $i$ in $\{1,\ldots,r\}$.
 In that case, we define the integral of $\varphi$ with respect
to the bounded Euler characteristic $\chi'$ by means of the formula
$$\int_{V}\varphi \,d\chi'=\sum_{i=1}^r a_i \chi'(\sigma_i) \quad \in A.$$
 If $\Gamma$ is a constructible subset of $V$, then we also write $$\int_{\Gamma}\varphi \,d\chi'=\int_{V}(\varphi \cdot \mathbf{1}_{\Gamma})\,d\chi'$$
 where $\mathbf{1}_\Gamma$ is the characteristic function of $\Gamma$.
\end{definition}

Integrals of constructible functions satisfy the following elementary Fubini property.

\begin{prop}\label{prop:fubini}
Let $A$ be an abelian group and let $f:V\to W$ be an affine linear map of finite dimensional affine spaces over $\Q$.
Let $\varphi:V\to A$ be a constructible function. Then the function
$$f_{\ast}\varphi:W\to A:w\mapsto \int_{f^{-1}(w)}\varphi \,d\chi'$$ is constructible, and
$$\int_{V}\varphi \,d\chi'=\int_{W}f_\ast\varphi \,d\chi'.$$
\end{prop}
\begin{proof}
By $A$-linearity of the integral, we may assume that $\varphi$ is the characteristic function of a polyhedron $\Gamma$ in $V$.
 Then the result follows at once from the fact that $f(\Gamma)$ and the fibers of $f$ are polyhedra in $W$, and that $\chi'$ assigns the value $1$ to every non-empty polyhedron.
\end{proof}

We can now formulate the main result of this paper.

\begin{theorem}[Motivic Fubini theorem for the tropicalization map]\label{thm:fubini}
 Let $Y$ be a variety over $\KK_0$.
Let $n$ be a positive integer and let $S$ be a semi-algebraic subset of $\mathbb{G}^n_{m,\KK_0}\times_{\KK_0}Y$.
Denote by $$\pi:\mathbb{G}^n_{m,\KK_0}\times_{\KK_0}Y\to \mathbb{G}^n_{m,\KK_0}$$ the projection morphism.
 Then the function
 $$(\trop \circ \pi)_*\mathbf{1}_S:\Q^n\to \gro^{\gal}(\Var_k):w\mapsto \Vol(S\cap (\trop\circ \pi)^{-1}(w))$$
 is constructible, and
 $$\Vol(S)=\int_{\Q^n}(\trop\circ \pi)_*\mathbf{1}_S \,d\chi'$$ in
 $\gro^{\gal}(\Var_k)$.
\end{theorem}

We will split up the proof of Theorem~\ref{thm:fubini} into two main steps. We can immediately make a first reduction.
 By additivity and Noetherian induction, we may assume that $Y$ is a closed subvariety of a split $\KK_0$-torus $T$.
 Denote by $N$ the cocharacter lattice of $T$. Applying Proposition~\ref{prop:fubini} to the function
 $$\varphi:\Q^n\times N_{\Q}\to \gro^{\gal}(\Var_k):w\mapsto \Vol(S\cap \trop^{-1}(w))$$ and the projection $f:\Q^n\times N_{\Q}\to \Q^n$,
 we see that it suffices to prove Theorem~\ref{thm:fubini} for the function  $\trop_{\ast}\mathbf{1}_{S}$ on $\Q^n\times N_{\Q}$, replacing $\GG^n_{m,\KK_0}$ by its product with $T$.
  Thus we may assume that $Y=\Spec \KK_0$.
 We split up the remainder of the proof into two steps.

\subsection{Step 1: the sch{\"o}n case}
We first consider the case of a sch{\"o}n integral closed subvariety $X_0$ of $\mathbb{G}^n_{m,\KK_0}$.
 The sch{\"o}nness condition means that $X=X_0\times_{\KK_0} \KK$ satisfies the following non-degeneracy condition: for every
 $a\in (\KK^{\times})^n$, the schematic closure of $a^{-1}X$ in $\GG^n_{m,R}$ is smooth over $R$. We denote this schematic closure by $\cX^a$.
 In \cite[3.11]{NPS} we have given a tropical formula for $\Vol(X)$ without taking the $\gal$-action into account. We will now explain how to refine this formula
 to keep track of the $\gal$-action.
 Actually, we will prove a slightly more general result, which applies to any semi-algebraic subset $S$ of $X_0$
 of the form $X(\Ka)\cap \trop^{-1}(\Gamma)$ where $\Gamma$ is a constructible subset of $\Q^n$.

Let $a\in (\KK^{\times})^n$ and set $w=\trop(a)$. Recall from Section \ref{ss:notation} that we denote by $\GG^{w}_{m,k}$ the $k$-torus $\GG^n_{m,k}$ endowed with the left $\gal$-action with weight vector $w$.
 Then $\cX^a_k$ is stable under the $\mu_d(k)$-action on $\GG^w_{m,k}$ and thus inherits a good action of $\gal$.
 It follows immediately from the definition that $\cX^a_k$, with its $\gal$-action, depends only on $w$, and not on $a$. It is called the {\em initial degeneration} of
 $X$ at $w$ and denoted by $\mathrm{in}_wX$.

  Let $\Sigma$ be a {\em $\Q$-admissible  tropical fan} for $X$ in $\RR^n\times \RR_{\geq 0}$ in the sense of \cite[12.1]{gubler}
  (henceforth, we will simply speak of a {\em tropical fan}). It defines a toric scheme  $\PP(\Sigma)$ over $R$.
   If we write $\cX$ for the schematic closure of $X$ in $\PP(\Sigma)$, then $\cX$ is proper over $R$ and the multiplication map
  $$m:\TT\times_R \cX\to \PP(\Sigma)$$
is faithfully flat. The condition that
 $X$ is sch\"{o}n is equivalent to the property that $m$ is smooth.
  The Galois action of $\gal$ on $\GG^n_{m,\KK}$ extends uniquely to $\PP(\Sigma)$, and $\cX$ is stable under this action.

Intersecting the cones of $\Sigma$ with $\QQ^n\times\{1\}$, we obtain a $\QQ$-rational polyhedral complex in $\QQ^n$ that we denote by $\Sigma_1$.
The support of $\Sigma_1$ is equal to $\trop(X(\Ka))$, by \cite[12.5]{gubler}.  For every cell $\gamma$ in $\Sigma_1$, we denote its relative interior by $\mathring{\gamma}$.
 We
write $X_{\gamma}$ for the semi-algebraic subset
   $$X(\Ka)\cap \trop^{-1}(\mathring{\gamma})$$
of $X$. As $\gamma$ ranges over the cells in $\Sigma_1$, the sets $X_{\gamma}$ form a partition of $X(\Ka)$. We denote by
$\cX_k(\gamma)$ the intersection of $\cX_k$ with the torus orbit of $\PP(\Sigma)_k$ corresponding to the cell $\gamma$. Then we can also write $X_\gamma$ as
$$X_\gamma=\spe^{-1}_{\cX}(\cX_k(\gamma))\cap X(\Ka).$$

\begin{lemma}\label{lemm:init} Let $X_0$ be a sch{\"o}n integral closed subvariety of $\mathbb{G}^n_{m,\KK_0}$ and let $\Sigma$ be a tropical fan for $X_0$
in $\RR^n\times \RR_{\geq 0}$.
 Let $\gamma$ be a cell of $\Sigma_1$. Then $\cX_k(\gamma)$ is smooth over $k$. If $w\in \mathring{\gamma}$ then the class of $\mathrm{in}_wX$ in $\gro^{\gal}(\Var_k)$
 is equal to $[\cX_k(\gamma)](\LL-1)^{\mathrm{dim}(\gamma)}$.
In particular, it only depends on $\gamma$, and not on $w$.
\end{lemma}
\begin{proof}
 Let $V$ be the $\Q$-linear subspace of $\Q^n$ generated by vectors of the form
 $w-w'$ with $w,w'$ in $\gamma$. We denote by $\widetilde{\TT}$ the split $R_0$-torus with cocharacter lattice
 $V\cap \ZZ^n$.
 Let $a$ be a point of $(\KK^{\times})^n$ such that $\trop(a)=w$.

 It is explained in the proof of \cite[3.10]{NPS} that $\cX^a_k$ is a trivial $\widetilde{\TT}_k$-torsor over $\cX_k(\gamma)$ in a natural way.
 Thus smoothness of $\cX_k(\gamma)$ follows from that of $\cX^a_k$.
 Moreover, an inspection of the proof reveals that the torsor structure is $\gal$-equivariant, where $\gal$ acts trivially on $\widetilde{\TT}_k$.
  This means that we can write $\cX^a_k$ as a product of line bundles on $\cX_k(\gamma)$ with the zero sections removed such that $\gal$ acts linearly on each factor -- see \cite[6.1.1]{BuNi}.
   The triviality of linear actions on vector spaces in $\gro^{\gal}(\Var_k)$ now implies that $$[\mathrm{in}_wX]=[\cX_k(\gamma)](\LL-1)^{\mathrm{dim}(\gamma)}.$$
\end{proof}
 Thus if $\gamma$ is a cell of $\Sigma_1$, no ambiguity arises from writing $[\mathrm{in}_\gamma X]$ for the class of $\mathrm{in}_wX$ in $\gro^{\gal}(\Var_k)$, where $w$ is any point in $\mathring{\gamma}$.

 \begin{prop}\label{prop:schonfubini}
Let $X_0$ be a sch{\"o}n integral closed subvariety of $\mathbb{G}^n_{m,\KK_0}$ and let $\Sigma$ be a tropical fan for $X_0$
in $\RR^n\times \RR_{\geq 0}$. Let $\Gamma$ be a constructible subset of $\Q^n$ and set $S=X(\Ka)\cap \trop^{-1}(\Gamma)$.
 Then
\begin{equation}\label{eq:schon}
\Vol(S)=\sum_{\gamma}\chi'(\Gamma\cap \mathring{\gamma})[\mathrm{in}_\gamma X]
\end{equation}
 in $\gro^{\gal}(\Var_k)$, where $\gamma$ runs through the set of cells in $\Sigma_1$.
In particular,
 $$\Vol(X_0)=\sum_{\gamma\,\mathrm{bounded}}(-1)^{\mathrm{dim}(\gamma)}[\mathrm{in}_\gamma X]$$ in $\gro^{\gal}(\Var_k)$, where the sum is taken over the bounded
 cells $\gamma$ of $\Sigma_1$.
 \end{prop}
\begin{proof}
 In order to deduce the formula for $\Vol(X_0)$ from equation \eqref{eq:schon},
  it suffices to observe that $\chi'(\mathring{\gamma})=(-1)^{\mathrm{dim}(\gamma)}$ when $\gamma$ is bounded, and $\chi'(\mathring{\gamma})=0$ otherwise
 -- see the proof of \cite[3.11]{NPS}.

 Therefore, we only need to prove the validity of \eqref{eq:schon}. Since both sides are additive in $\Gamma$ and invariant under refinement of the fan $\Sigma$, we may assume that
 $\Gamma=\mathring{\gamma}$ for some cell $\gamma$ in $\Sigma_1$. Then we must show that
 $\Vol(X_\gamma)=\chi'(\mathring{\gamma})[\mathrm{in}_\gamma X]$.
 We will follow a similar construction as in the proof of \cite[3.10]{NPS}, but we will need to refine it to take the $\gal$-action into account.

  We fix a point $w$ in $\mathring{\gamma}$ and set $t^{w}=(t^{w_{1}},\ldots,t^{w_{n}})\in (\KK^{\times})^n$.
 Let $\TT=\GG^n_{m,R}$ and $T=\GG^n_{m,K}$.
  We denote by $\TT^w$ the torus $\TT$ endowed with the $\gal$-action with weight vector $w$, and by $T^w$ its generic fiber.
   Set $$\mathring{\gamma}_w=\{v\in \Q^n\,|\,w+v\in \mathring{\gamma}\}$$ and
   let $V$ be the $\Q$-linear subspace of $\Q^n$ generated by $\mathring{\gamma}_w$.  We denote by $\widetilde{\TT}$ the split $R$-torus with cocharacter lattice
 $V\cap \ZZ^n$. This is a subtorus of $\TT$. We write $\widetilde{T}$ for the generic fiber of $\widetilde{\TT}$, and $\widetilde{T}_{\gamma}$ for the inverse image of
 $\mathring{\gamma}_w$ under the tropicalization map $\widetilde{T}(\Ka)\to V\cap \Q^n$.
   The quotient  $\TT_k/\widetilde{\TT}_k$ acts freely and transitively on the torus orbit
 $O(\gamma)$ of $\PP(\Sigma)_k$ corresponding to the cell $\gamma$. If we take the specialization of $t^{w}$ as base point, then we obtain a $\gal$-equivariant isomorphism between
  $O(\gamma)$ and
 $\TT^{w}_k/\widetilde{\TT}_k$, where we let $\widetilde{\TT}$ act on $\TT^{w}$ by multiplication.

 The action of $\gal$ on $\TT^w_k/\widetilde{\TT}_k$ factors through a free action of $\mu_d(k)$ for some $d>0$. Thus  we can find locally on $\TT^w_k/\widetilde{\TT}_k$
 a $\gal$-equivariant \'etale morphism to $\AA^r_k$ equipped with the trivial $\gal$-action, for some $r\geq 0$.
  Taking the base change to $R$, this implies that there exists around each point of $\cX_k(\gamma)$ a $\gal$-stable open neighbourhood $\cU$ in $\TT^{w}/\widetilde{\TT}$ that
 admits an \'etale $\gal$-equivariant morphism $$h:\cU\to \AA^r_{R}.$$ Moreover, since $\cX_k(\gamma)$ is smooth,
 we can arrange that $\cX_k(\gamma)\cap \cU=h^{-1}(\AA^s_k)$ for some $0\leq s\leq r$.
   We set $\cY=\cU\times_{\AA^r_{R}}\AA^s_{R}$ and we endow it
 with the diagonal action of $\gal$. We will construct
 a semi-algebraic bijection
 $$\psi:\spe^{-1}_{\cX}(\cY_k)\cap X(\KK)\to  \widetilde{T}_{\gamma}\times\cY(R)$$
that is defined over $\KK_0$ and commutes with the specialization maps to $\cY_k$.
 The result then follows from the fact that $\Vol(\cY(\Ra))=[\cU\cap \cX_k(\gamma)]$ and
 $\Vol(\widetilde{T}_{\gamma})=\chi'(\mathring{\gamma})(\LL-1)^{\mathrm{dim}(\gamma)}$ by Theorem~\ref{thm:HK}.

 By the henselian property of $R$, the linear projection $\pi:\AA^r_R \to \AA^s_R$ induces
 a semi-algebraic retraction $$\rho:\spe^{-1}_{\cU}(\cY_k)\to \cY(\Ra),$$
 which can be described in the following way. Let $\cZ$  be the inverse image in $\cU\times_R \cY$ of the graph of $\pi$.
 Let $\Delta$ be the image of the diagonal embedding of $\cY_k$ into $\cU_k\times \cY_k$.
  Then the graph of $\rho$ is equal to $\spe^{-1}_{\cZ}(\Delta)$.
 Thus $\rho$ is defined over $\KK_0$, by Example~\ref{exam:descent}.

 We also choose
 a $\ZZ$-linear retraction of $(V\cap \ZZ^n)\to \ZZ^n$, and we denote by $p_0:T\to \widetilde{T}$ the morphism of
 tori associated with the morphism of cocharacter lattices $\ZZ^n\to (V\cap \ZZ^n)$.
 We write $$p:T^{w}\to \widetilde{T}$$ for the composition of the morphism
 $T^{w}\to T$ given by multiplication by $t^{-w}$, and the projection morphism $p_0$. Then $p$ is $\gal$-equivariant, and thus defined over $\KK_0$.
Finally, we write $q$ for the projection map $\TT^w\to \TT^w/\widetilde{\TT}$.

 Now we consider the semi-algebraic map
 $$\psi:\spe^{-1}_{\cX}(\cY_k)\cap X(\Ka)\to \widetilde{T}_\gamma\times \cY(\Ra):x\mapsto (p(x),\rho(q(p(x)^{-1}\cdot x))).$$
 It is defined over $\KK_0$.  Note that $p(x)$ indeed lies in $\widetilde{T}_{\gamma}$ because $\trop(p(x))$ is the image of $\trop(x)-w$ under the projection
  $\QQ^n\to (V\cap \QQ^n)$. We claim that $\psi$ is bijective. To prove this, it suffices to show that the map
  $$\varphi_a:(\cX^{a}\times_{\TT/\widetilde{\TT}}\cU)(\Ra)\to \widetilde{\TT}(\Ra)\times \cY(\Ra):x\mapsto (p_0(x), \rho(q(x)))$$
  is bijective for every $a$ in $\widetilde{T}_\gamma$. This can be done in the same way as in the proof of \cite[3.10]{NPS}:
   the morphism $p_0$ induces a splitting $\TT\cong \widetilde{\TT}\times_R (\TT/\widetilde{\TT})$, and under this splitting
  the restriction of the morphism
  $$(p_0,\pi \circ h \circ q):\cX^{a}\times_{\TT/\widetilde{\TT}}\cU\to  \widetilde{\TT}\times_R \AA^s_R$$
to the special fibers coincides with the \'etale morphism
$$\mathrm{Id}\times h_k:\widetilde{\TT}_k\times_k \cY_k\to \widetilde{\TT}_k\times_k \AA^s_k.$$
 Now the henselian property of $R$ implies that $\varphi_a$ is a bijection.
    \end{proof}

\begin{cor}\label{cor:schon}
Let $X_0$ be a sch{\"o}n integral closed subvariety of $\GG^n_{m,K_0}$ and let $\Gamma$ be a constructible subset of $\Q^n$. Then Theorem~\ref{thm:fubini} holds for
$S=X_0(\Ka)\cap \trop^{-1}(\Gamma)$.
\end{cor}
\begin{proof}
 Let $a$ be an element in $(\KK^{\times})^n$ and write $\trop(a)=w$. We have
 $\trop^{-1}(w)=\cX^a(\Ra)$. Thus Theorem~\ref{thm:HK} implies that $\Vol(\trop^{-1}(a))=[\cX^a_k]$. If we denote by $\gamma$ the unique cell of $\Sigma_1$ that contains $a$ in its relative interior,
we know by Lemma~\ref{lemm:init} that $[\cX^a_k]=[\mathrm{in}_\gamma X]$.
Thus the function $w\mapsto \Vol(\trop^{-1}(w))$ is constant with value $[\mathrm{in}_\gamma X]$ on the
relative interior $\mathring{\gamma}$
 of each cell $\gamma$ of $\Sigma_1$. Hence, Theorem~\ref{thm:fubini} follows from formula \eqref{eq:schon}.
 \end{proof}

\subsection{Intrinsic torus embeddings}\label{ss:intrinsic}
Before we prove the general case of Theorem~\ref{thm:fubini}, we collect some results on intrinsic tori that
will be used in the proof.

\begin{definition}
Let $F$ be a field.  If $T$ and $T'$ are split $F$-tori, then a {\em monomial morphism} $T'\to T$
  is a morphism of $F$-varieties that can be written as the composition of a morphism of tori followed by a translation by an element in $T(F)$.
  \end{definition}

Let $T'\to T$ be a monomial morphism of $K_0$-tori. If we denote by $N$ and $N'$ the cocharacter lattices of $T$ and $T'$, respectively, then there exists a unique
 integral affine linear map $\nu:N'\to N$
such that the following diagram commutes:
 $$\begin{CD}
 T'(\Ka)@>\trop>> N'_\Q
\\ @VVV @VV\nu_{\Q} V
\\ T(\Ka)@>>\trop> N_\Q
 \end{CD}$$

Let $F$ be a field. A variety $U$ over $F$ is called {\em very affine} if it admits a closed embedding
into a split $F$-torus. Then the group $M_U=\mathcal{O}(U)^{\times}/F^{\times}$ is a free $\ZZ$-module of finite rank.
  The corresponding split $F$-torus $T'=\Spec F[M_U]$ is called the intrinsic torus of $U$. The choice of a section
  $s:M_U\to \mathcal{O}(U)^{\times}$ determines a closed embedding $f:U\to T'$, which we call an intrinsic torus embedding. Changing $s$ amounts to composing this embedding
  with a translation by a point in $T'(F)$. Now let $h:U\to T$ be a locally closed embedding into a split $F$-torus $T$, with character lattice $M$.
   This embedding induces a morphism of lattices $M\to M_U$ and hence a morphism of tori $g:T'\to T$.
 The morphisms $h$ and $h'=g\circ f$ correspond to two homomorphisms $\psi,\psi':M\to \mathcal{O}(U)^{\times}$ that coincide
 after composition with the projection map $\mathcal{O}(U)^{\times}\to M_U$.  This means that the image of the quotient $\psi/\psi'$ is contained in $F^{\ast}$, and
 $h$ and $h'$
 coincide up to translation by the point $a$ in $T(F)$ defined by $\psi/\psi'$. Composing the morphism $T'\to T$
 with the translation by $a$, we obtain a monomial morphism of tori $T'\to T$ such that the restriction to $U$ is an isomorphism onto $U$.

\subsection{Step 2: the general case}
We will now prove Theorem~\ref{thm:fubini} by reducing it to the case that was treated in Proposition~\ref{prop:schonfubini}.
 We will achieve this reduction by partitioning any semi-algebraic set into pieces to which  Proposition~\ref{prop:schonfubini} can be applied.

 \begin{lemma}\label{lemm:schon}
 Let $T$ be a split algebraic torus over $\KK_0$ and let $X$ be a subvariety of $T$.
 Then we can find:
 \begin{itemize}
 \item a partition of $X$ into subvarieties $U_1,\ldots,U_r$,
 \item for every $i$ in $\{1,\ldots,r\}$, a sch{\"o}n closed embedding $U_i\to T_i$ where $T_i$ is a split torus over $\KK_0$,
 \item for every $i$ in $\{1,\ldots,r\}$, a monomial morphism of tori $T_i\to T$  such that the restriction of $T_i(\Ka)\to T(\Ka)$ to $U_i(\Ka)$ is a bijection onto $U_i(\Ka)$ for every $i$.
 \end{itemize}
  \end{lemma}
  \begin{proof}
  By \cite[7.10]{LQ}, we can find a very affine non-empty open subvariety $U$ of $X$ such that every intrinsic torus embedding  of $U$  is sch{\"o}n.
   We choose such an intrinsic embedding $U\to T'$ .
As explained in Section \ref{ss:intrinsic}, the embedding $U\to T$ gives rise to
a monomial morphism of tori $T'\to T$ such that the restriction to $U$ is an isomorphism onto $U$.
 Now the result follows from Noetherian induction on $X$.
  \end{proof}

 \begin{prop}\label{prop:schon}
 Let $T$ be a split $\KK_0$-torus and let $S$ be a semi-algebraic subset of $T$. Then there exists
 a finite partition of $S$ into semi-algebraic subsets $S'$ such that, for each $S'$, there exist:
\begin{itemize}
\item a subvariety $U$ of $T$ such that $S'$ is contained in $U(\Ka)$,
\item a sch{\"o}n closed embedding of $U$ into a split $\KK_0$-torus $T'$ with cocharacter lattice $N'$
such that
 $S'$ is of the form $U(\Ka)\cap \trop^{-1}(\Gamma)$ for some constructible subset $\Gamma$ in $N'_\QQ$,
\item  a monomial morphism $T'\to T$ such that the restriction of $T'(\Ka)\to T(\Ka)$ to $S'\subset T'(\Ka)$
 is a bijection onto $S'\subset T(\Ka)$.
\end{itemize}
 \end{prop}
 \begin{proof}
Let $N$ be the cocharacter lattice of $T$. Partitioning $S$, we may assume that there exist a subvariety $X$ of $T$ and invertible regular functions
$$f_1,g_1,\ldots,f_r,g_r$$ on $X$ such that $S$ is given by
$$S=\{x\in X(\Ka)\,|\,\val(f_i(x))\Box_i \val(g_i(x))\mbox{ for each }i\}$$
where $\Box_i$ is either $\leq$ or $<$.  If we re-embed $X$ via the morphism
$$X\to T\times_{\KK_0}\GG^{2r}_{m,\KK_0}: x\mapsto (x,f_1(x),g_1(x),\ldots,f_r(x),g_r(x))$$
then $S$ is of the form $X(\Ka)\cap \trop^{-1}(\Gamma)$ where $\Gamma$ is a constructible subset of
$N_{\Q}\times \QQ^{2r}$. Thus we may assume that $S$ is of the form $X(K)\cap \trop^{-1}(\Gamma)$ where $\Gamma$ is a constructible subset of
$N_{\Q}$.

Let $U$ be a subvariety of $X$ and $U\to T'$ a sch{\"o}n closed embedding into a split $\KK_0$-torus $T'$. Assume that there exists a
monomial morphism of tori $T'\to T$ such that the restriction of $T'(\Ka)\to T(\Ka)$ to $U(\Ka)$ is a bijection onto $U(\Ka)$.
 Then the restriction of $T'(\Ka)\to T(\Ka)$ to $U(\Ka)\cap S$ is a bijection onto $U(\Ka)\cap S$.
  Moreover, if we denote by $\nu:N'\to N$ the integral affine linear map of cocharacter lattices associated with $T'\to T$, then
  $$ S\cap U(\Ka)=\trop^{-1}(\nu_{\Q}^{-1}(\Gamma)).$$
 The result now follows from the fact that we can partition $X$ into subvarieties $U$ of this form, by Lemma~\ref{lemm:schon}.
 \end{proof}

Now we can finish the proof of Theorem~\ref{thm:fubini}. By additivity, we may assume that there exist
 $U$, $T'$ and $\Gamma$ as in Proposition~\ref{prop:schon} for $S'=S$. Let $\nu:N'\to N$ be the integral affine linear map of cocharacter lattices associated with
  the monomial morphism $T'\to T$.
 We consider the commutative diagram
 $$\begin{CD}
 T'(\Ka)@>\trop'>> N'_\Q
\\ @VVV @VV\nu_{\Q} V
\\ T(\Ka)@>>\trop> N_\Q
 \end{CD}$$
 where we write $\trop'$ to distinguish between the two tropicalization maps.
   By the sch{\"o}n case of the Fubini theorem (Proposition~\ref{prop:schonfubini}), we know that Theorem~\ref{thm:fubini} holds
 for the embedding of $S\cap (\trop')^{-1}(\Gamma')$ into $T'(\Ka)$, for every constructible subset $\Gamma'$ of $N'_{\Q}$.
  Taking $\Gamma'=N'_{\Q}$ we see that $\trop'_{\ast}\mathbf{1}_S$ is constructible and
  $$\Vol(S)=\int_{N'_{\Q}} \trop'_{\ast}\mathbf{1}_S\,d\chi'.$$
  Taking for $\Gamma'$ a fiber of the affine linear map $\nu_{\Q}$, we also find that
 $$\trop_{\ast}\mathbf{1}_S=(\nu_{\Q})_{\ast}\trop'_{\ast}\mathbf{1}_S.$$
 Now it follows from  Proposition~\ref{prop:fubini} that
 $\trop_{\ast}\mathbf{1}_S$ is constructible and
  $$\Vol(S)=\int_{N_{\Q}} \trop_{\ast}\mathbf{1}_S\,d\chi'.$$
  This concludes the proof of Theorem~\ref{thm:fubini}. \hfill \qed

\subsection{Properties of the volumes of tropical fibers}\label{ss:tropfib}
In order to apply Theorem~\ref{thm:fubini} to concrete problems, it is useful to have some information about the shape
of the constructible decomposition of $\Q^n$ on which the tropical fibers have piecewise constant volumes.
We will prove two statements that will be important in the proofs of Conjectures A and B. We start with a basic
proposition on torus-equivariance.

\begin{prop}\label{prop:torstable}
 We keep the notations of Theorem~\ref{thm:fubini}.
 Let $\Ka^{\times}$ act on $(\Ka^{\times})^n$ with weight vector $w\in \ZZ^n$, and trivially on $Y(\Ka)$.
 Assume that $S$ is stable under this action.
 Then the function $$\varphi:\Q^n\to \gro^{\gal}(\Var_k):v\mapsto \Vol(S\cap (\trop\circ \pi)^{-1}(v))$$ is constant along the line $v+\Q w$ for every $v$ in $\Q^n$.
\end{prop}
\begin{proof}
The function $\varphi$ is constructible, by Theorem~\ref{thm:fubini}, and it is periodic with period
 $w$
 because of the semi-algebraic bijection
 $$S\cap (\trop\circ \pi)^{-1}(v)\to S\cap (\trop\circ \pi)^{-1}(v+w)$$
 defined by the action of $t\in K_0^{\times}$,  for every $v$ in $\Q^{n}$.
  Thus the restriction of $\varphi$ to every line of the form $v+\Q w$ in $\Q^{n}$ is both constructible and periodic, which is only possible
 if it is constant.
\end{proof}

To formulate the second result, we need to make some preparations. Let $n$ be a positive integer and let $w$ be an element of $\Q^n$.
We denote by
$\GG^w_{m,k}$
 the $k$-torus $\GG^n_{m,k}$
 endowed with the left $\gal$-action with weight vector $w$ (see Section \ref{ss:notation}).
 We set $\GG^w_{m,K}=\GG^w_{m,k}\times_k K$, endowed with the diagonal $\gal$-action.
  We define the tropicalization map
  $$\trop:\GG^w_{m,K}(\Ka)\to \Q^n$$ by ignoring the $\gal$-action on $\GG^w_{m,K}$.
 Let $Y$ be a $k$-variety with trivial $\gal$-action.
 We say that a semi-algebraic subset $S$ of $\GG^w_{m,K}\times_k Y$ is {\em defined over $k$} if we can write it as a finite Boolean combination of
 sets of the form   $$\{x\in (\GG^w_{m,K}\times_k U)(\Ka)\,|\,\val(f(x))\leq \val(g(x))\}$$
 where $U$ is an affine open subvariety of $Y$ and $f$ and $g$ are regular functions on $\GG^w_{m,k}\times_k U$
 that are invariant under the $\gal$-action.
  Then $S$ is also defined over $K_0$, that is, it is a semi-algebraic subset of the $K_0$-scheme of finite type
 $(\GG^w_{m,K}/\gal)\times_k Y$. The torus $\GG^n_{m,K}$ acts on $\GG^w_{m,K}$ by multiplication from the left, and the multiplication morphism
 $$\GG^n_{m,K}\times_{K} \GG^w_{m,K}\to \GG^w_{m,K}$$ is $\gal$-equivariant and descends to a morphism
 $$\GG^n_{m,K_0}\times_{K_0} (\GG^w_{m,K}/\gal)\to \GG^w_{m,K}/\gal$$
 that makes $\GG^w_{m,K}/\gal$ into a $\GG^n_{m,K_0}$-torsor.
 We can trivialize this torsor by means of the $\gal$-equivariant isomorphism
 $$\GG^n_{m,K}\to \GG^w_{m,K}:(x_1,\ldots,x_n)\mapsto (t^{w_1}x_1,\ldots,t^{w_n}x_n)$$
that maps the identity of $\GG^n_{m,K}$ to the $\gal$-fixed point $(t^{w_1},\ldots,t^{w_n})$ of $\GG^w_{m,K}$.

\begin{rem}
The expression ``defined over $k$'' is a slight abuse of terminology; more precisely, what we are considering here are twisted forms over $K_0$ of semi-algebraic sets in $\GG^n_{m,K}\times_k Y$ defined over $k$.
\end{rem}

\begin{prop}\label{prop:fan}
Let $w$ be an element of $\Q^n$, for some $n>0$. Let $Y$ be a $k$-variety with trivial $\gal$-action and let $S$ be a semi-algebraic subset of
$\GG^w_{m,K}\times_k Y$ that is defined over $k$.
 Denote by $$\pi:\GG^w_{m,K}\times_k Y\to \GG^w_{m,K}$$ the projection morphism.
Then there exists a complete fan  in $\Q^n$
such that the function
$$\varphi:\Q^n\to \gro^{\gal}(\Var_k):v\mapsto \Vol(S\cap (\trop\circ \pi)^{-1}(v))$$ is constant on
the relative interior of each cone.
\end{prop}
\begin{proof}
We can make a similar reduction as at the beginning of the proof of Theorem~\ref{thm:fubini}: by additivity we may assume that $Y$ is
a closed subvariety of a $k$-torus with trivial $\gal$-action, and by absorbing this torus into $\GG^w_{m,k}$ we can reduce to the case where
$Y=\Spec k$. We first deal with the sch{\"o}n case.
 An integral closed subvariety $X$ of $\GG^n_{m,k}$ is called sch{\"o}n if $X\times_k K$ is sch{\"o}n in $\GG^n_{m,K}$.
  Let $X$ be a sch\"on integral closed subvariety of $\GG^n_{m,k}$ that is stable under the $\gal$-action.
 Let $\Gamma$ be a constructible subset of $\Q^n$ that is stable under
scalar multiplication with elements in $\Q_{>0}$, and let $S=X(\Ka)\cap \trop^{-1}(\Gamma)$.
 Let $\Sigma'_1$ be a tropical fan for $X$ in $\R^n$ in the sense of \cite{tevelev}.
  We may assume that $\Gamma$ is a union of relatively open cones in $\Sigma'_1$.
 Set $\Sigma_1=-w+\Sigma'_1$ and let
 $\Sigma$ be the fan over $\Sigma_1\times \{1\}$ in $\R^n\times \R_{\geq 0}$. This is a tropical fan
 for the sch{\"o}n subvariety $X'=t^{-w}X$ of $\GG^n_{m,K}$, and $X'$ is defined over $K_0$.
 Multiplication with $t^{-w}$ yields an isomorphism between $X$ and
 $X'$ defined over $K_0$ that commutes with the tropicalization maps up to translation by $w$. Thus we can deduce Proposition~\ref{prop:fan} for $S$
 by applying Proposition~\ref{prop:schonfubini} to the semi-algebraic subset $t^{-w}S$ of $X'$.

    To prove the general case, we can proceed in a similar way as in step 2 of the proof of Theorem~\ref{thm:fubini}.
 It is sufficient to show that, for every $\gal$-stable integral subvariety $X$ of $\GG^w_{m,k}$, we can find
 a $\gal$-stable dense subvariety $U$ of $X$, a $\gal$-equivariant sch{\"o}n closed embedding $U\to \GG^{w'}_{m,k}$ for some $w'\in \Q^{n'}$, and a
a $\gal$-equivariant monomial morphism of tori $\GG^{w'}_{m,k}\to \GG^w_{m,k}$
that induces an isomorphism $U\to U$. Then the rest of the proof of Theorem~\ref{thm:fubini} immediately carries over to our setting.

We set $T=\GG^w_{m,k}$. The quotient of $T$ by the action of $\gal$ is a split $k$-torus $\widetilde{T}$, and the quotient map $T\to \widetilde{T}$ is a Kummer finite \'etale cover
 of degree $d$, the smallest positive integer such that $d\cdot w$ lies in $\Z^n$. Let $\widetilde{X}$ be the image of $X$ in $\widetilde{T}$.
  By \cite[1.4]{LQ}, we can find a dense very affine open subvariety $\widetilde{U}$ in $\widetilde{X}$ such that every intrinsic torus embedding of $\widetilde{U}$  is sch{\"o}n.
   We choose such an intrinsic embedding $\widetilde{U}\to \widetilde{T}'$.  By the discussion in Section \ref{ss:intrinsic}, the embedding of $\widetilde{U}$ into $\widetilde{T}'$ induces a monomial morphism of tori $f:\widetilde{T}'\to \widetilde{T}$  that restricts to an isomorphism from $\widetilde{U}$ onto $\widetilde{U}$.
  Let $U=\widetilde{U}\times_{\widetilde{T}}T$ be the inverse image of $\widetilde{U}$ in $T$; it is a $\gal$-stable dense open subvariety of $X$.
  Set $T'= \widetilde{T}'\times_{\widetilde{T}} T$; then we can endow $T'$ with the structure of a split $k$-torus such that $T'\to \widetilde{T}'$ is a morphism of tori and $T'\to T$ is a monomial morphism.
   The torus $T'$ inherits a good $\gal$-action from $T=\GG^w_{m,k}$ such that $\widetilde{T}'=T'/\gal$, and $T'$ is of the form $\GG^{w'}_{m,k}$ for some finite tuple $w'$ of rational numbers.

  The closed embedding $\widetilde{U}\to \widetilde{T}'$ induces a $\gal$-equivariant closed embedding
  $g:U\to T'$ by base change.
   The projection morphism $T'\to \widetilde{T}'$ induces a finite \'etale morphism of split $R$-tori $h:T'\times_k R\to \widetilde{T}'\times_k R$ by base change.
  For every point $a$ of $T(K)$, the closure of $a^{-1}(U\times_k K)$ in $T'\times_k R$ is the inverse image under $h$ of the closure of $(h(a))^{-1}(\widetilde{U}\times_k K)$ in $\widetilde{T}'\times_k R$
 because $h$ is finite and flat. Restricting $h$ to the special fibers, we see that the initial degeneration of $U\times_k K$ with respect to $a$ is a finite \'etale cover of the initial degeneration of $\widetilde{U}\times_k K$ with respect to $h(a)$. The latter is smooth because $\widetilde{U}$ is sch\"on, and it follows that all the initial degenerations of the embedding $U\times_k K\to T'\times_k K$ are smooth, as well. In other words, the closed embedding $U\to T'$ is sch\"on.
    This concludes the proof.
  \end{proof}

\begin{rem}
A more intuitive explanation for the fan structure in Proposition~\ref{prop:fan} is that we can reparameterize points in $S$ by substituting $t$ by $t^q$
 for any positive rational number $q$. Since we need to keep track of the $\gal$-action,
 writing down a proof along these lines is somewhat tedious,
 which is why we have opted for the cleaner argument in the proof of Proposition~\ref{prop:fan}.
\end{rem}

\section{Proofs of Conjectures A and B}\label{sec:appli}
\subsection{The conjecture of Davison and Meinhardt}
As a first application, we prove Conjecture A from the introduction. Using the comparison result in Corollary~\ref{cor:compar},
 the following theorem is a strengthening of Conjecture A (the strengthening being that we do not invert $\LL$).

\begin{thm}\label{thm:DM}
Let $Y$ be a smooth $k$-variety endowed with the trivial $\GG_{m,k}$-action, and let $\GG_{m,k}$ act on $\AA^n_k$ with weights $w_1,\ldots,w_n> 0$.
 Set $U= \AA^n_k \times_k Y$ and let $$f:U\to \AA^1_k=\Spec k[t]$$ be a morphism of $k$-varieties that is $\GG_{m,k}$-equivariant,
 where $\GG_{m,k}$ acts on $\AA^1_k$ with weight $d>0$.
    If we set $\cV=U\times_{k[t]}R_0$, then $$\Vol(\cV(\Ra))=[f^{-1}(1)]$$ in $\gro^{\gal}(\Var_k)$, where the $\gal$-action on $f^{-1}(1)$ factors through $\mu_d(k)$ and is given by
      $$\mu_d(k)\times f^{-1}(1)\to f^{-1}(1):(\zeta,(x_1,\ldots,x_n,y))\mapsto (\zeta^{w_1}x_1,\ldots,\zeta^{w_n}x_n,y).$$
\end{thm}
\begin{proof}
Set $\cW=f^{-1}(1)\times_k R$ and endow it with the diagonal $\gal$-action. Then $\cW$ is smooth over $R$ and $\Vol(\cW(\Ra))=[f^{-1}(1)]$ in $\gro^{\gal}(\Var_k)$
by the definition of the motivic volume.
 The map
$$\cV(\Ka)\to \cW(\Ka):(x_1,\ldots,x_n,y)\mapsto (t^{-w_1/d}x_1,\ldots,t^{-w_n/d}x_n,y)$$
is a semi-algebraic bijection that is defined over $\KK_0$, and identifies $\cV(\Ra)$ with a subset of $\cW(\Ka)$ that contains $\cW(\Ra)$.
It suffices to show that the motivic volume of the complement of $\cW(\Ra)$ in the image of $\cV(\Ra)$ vanishes.
This complement is given by
$$S=\{(x_1,\ldots,x_n,y)\in \cW(\Ka)\,|\,\val(x_i)\geq -w_i/d\mbox{ for all }i,\,\val(x_i)<0\mbox{ for some }i,\,y\in Y(\Ra)\}.$$
 By induction on $n$, it is enough to prove that
 $$S^o=S\cap ((\Ka^{\times})^n\times Y(\Ra))$$
has motivic volume $0$.

We can view $\cW(\Ka)\cap ((\Ka^{\times})^n\times Y(\Ra))$
as a semi-algebraic subset of $\GG^w_{m,K}\times_k Y$ that is
 defined over $k$, in the sense of Section \ref{ss:tropfib}.
 We denote by $\pi$ the projection map
 $$\pi: (\Ka^{\times})^n\times Y(\Ra)\to (\Ka^{\times})^n$$
and we consider the function
$$\varphi:\Q^n\to \gro^{\gal}(\Var_k):v\mapsto \Vol(\cW(\Ka)\cap (\trop \circ \pi)^{-1}(v)).$$
 Set
 $$\Gamma=\{v\in \Q^n\,|\,v_i\geq -w_i/d\mbox{ for all }i,\,v_i<0 \mbox{ for some }i\}.$$
 By Theorem~\ref{thm:fubini}, the function $\varphi$ is constructible, and we can compute the motivic volume of $S^o$ as
 $$\Vol(S^o)=\int_{\Gamma}\varphi \,d\chi'.$$
  By Proposition~\ref{prop:fan}, we can find a complete fan in $\Q^n$ such that $\varphi$ is constant on every relatively open cone $\sigma$ in this fan.
 If we denote by $(\sigma\cap \Gamma)_{\R}$ the subset of $\R^n$ associated with $\sigma\cap \Gamma$,
  then the intersection of $(\sigma\cap \Gamma)_{\R}$ with any box $[-r,r]^n$, $r\in \R_{>0}$, is homeomorphic to $[0,1)\times \partial$
 where $\partial$ denotes the boundary of $(\sigma\cap \Gamma)_{\R}\cap [-r,r]^n$. Since the compactly supported Euler characteristic of $[0,1)$ vanishes, it follows that
 $\chi'(\sigma\cap \Gamma)=0$ for all $\sigma$, so that $\Vol(S^o)=0$.
\end{proof}
\if false
\begin{rem}
If we denote by $\rho:(K^{\times})^n\times Y(K)\to (K^{\times})^n$ the projection map, then
our proof also shows
that $$\Vol(\cV(R)\cap ((K^{\times})^n\times Y(K)))=\Vol(\cV(K)\cap (\trop \circ \rho)^{-1}(w/d)).$$
Thus the entire contribution from $\mathbb{G}^n_{m,k}\times Y$ to the motivic nearby fiber of $f$ comes
from the $K$-points that tropicalize to $w/d$. Note that the intersection of $f^{-1}(1)$ with $\GG^n_{m,k}$ is precisely
the initial degeneration of $f$ at $w/d$.
\end{rem}
\fi

\subsection{The integral identity of Kontsevich and Soibelman}
As a second application, we give a short proof of the integral identity conjecture of Kontsevich and Soibelman (Conjecture B in the introduction).
L\^e Quy Thuong proved this conjecture in \cite{thuong}, also using Hrushovski-Kazhdan motivic integration.
  Our comparison statement in Corollary~\ref{cor:compar} and our Fubini theorem for the tropicalization map allow us to substantially
  simplify the proof and avoid the inversion of $\LL$. We also generalize the statement by allowing
  arbitrary positive weights on $\AA^{d_1}_k$ and arbitrary negative weights on $\AA^{d_2}_k$, and replacing the factor $\AA^{d_3}_k$ by any $k$-variety with trivial
  $\GG_{m,k}$-action.
      By Corollary~\ref{cor:compar}, the following theorem is a generalization of Conjecture B.

\begin{thm}\label{thm:KS}
 Let $Z$ be a $k$-variety with trivial $\GG_{m,k}$-action and let $p$ be a closed point in $Z$.
Let $d_1$ and $d_2$ be nonnegative integers and let $\GG_{m,k}$ act diagonally on $$U=\AA^{d_1}_k\times_k \AA^{d_2}_k\times_k Z$$
with positive weights on the first factor, with negative weights on the second factor, and trivially on $Z$.
 Let $$f:U\to \AA^1_k=\Spec k[t]$$ be a dominant morphism
that is $\GG_{m,k}$-equivariant, where $\GG_{m,k}$ acts trivially on the target $\AA^1_k$, and such that $f(0,0,p)=0$.

   We view $\AA^{d_1}_k$ and $Z$ as closed subvarieties of $U$ via the embeddings $x\mapsto (x,0,p)$ and $z\mapsto (0,0,z)$, respectively.
 Then $f$ vanishes on $\AA^{d_1}_k$.
 We set $\cV=U\times_{k[t]}R_0$  and $\cW=Z\times_{k[t]}R_0$.
 Then
$$\Vol(\spe^{-1}_{\cV}(\AA^{d_1}_k))=\LL^{d_1}\Vol(\spe^{-1}_{\cW}(p))$$
 in $\gro^{\gal}(\Var_k)$.
\end{thm}
\begin{proof}
The vanishing of $f$ on $\AA^{d_1}_k$ follows from $\GG_{m,k}$-equivariance, together with the fact that $f(0,0,p)=0$.
 We partition $S=\Vol(\spe^{-1}_{\cV}(\AA^{d_1}_k))$ into the semi-algebraic sets
\begin{eqnarray*}
S_0&=&\{v=(x,y,z)\in \cV(\Ra)\,|\,\spe_{\cV}(v)\in \AA^{d_1}_k\mbox{ and }y\neq 0\},
\\ S_1&=&\{v=(x,y,z)\in \cV(\Ra)\,|\,\spe_{\cV}(v)\in \AA^{d_1}_k\mbox{ and }y= 0\}.
\end{eqnarray*}
We first compute the volume of $S_1$. By the $\GG_{m,k}$-equivariance of $f$, the value of $f(x,0,z)$ only depends on $z$, and we can write
 $S_1=\spe^{-1}_{\cW}(p)\times \Ra^{d_1}$.
 The second factor has motivic volume $\LL^{d_1}$, so that
 $$\Vol(S_1)= \LL^{d_1}\Vol(\spe^{-1}_{\cW}(p)).$$

Therefore, it suffices to show that the motivic volume of $S_0$ vanishes. By additivity, it is enough to prove this after replacing $S_0$
by $S'_0=S_0\cap (O\times Z(\Ra))$, where $O$ is any $(\Ka^{\times})^{d_1+d_2}$-orbit in $\Ka^{d_1}\times \Ka^{d_2}$.
 If $O$ is contained in $\Ka^{d_1}\times \{0\}$ then $S'_0$ is empty.
 Thus, by induction on $d_1+d_2$, we may assume that $d_2>0$ and $$O=(\Ka^{\times})^{d_1}\times (\Ka^{\times})^{d_2}.$$

 We write $$\pi:U\to \AA_k^{d_1}\times \AA^{d_2}_k$$ for the projection onto the first two factors.
 Consider the semi-algebraic set $$S'=\{(x,y,z)\in \cV(\Ka)\,|\,(x,y)\in O,\ z\in Z(\Ra),\ \spe_{Z\times_k R_0}(z)=p\}$$
  and the function $$\varphi:\Q^{d_1}\times \Q^{d_2}\to \gro^{\gal}(\Var_k):v\mapsto \Vol(S'\cap (\trop\circ \pi)^{-1}(v)).$$
   The set $S'_0$
   is the subset of $S'$ defined by the conditions $\val(x_i)\geq 0$ and $\val(y_j)>0$ for $i\in \{1,\ldots,d_1\}$ and $j\in \{1,\ldots,d_2\}$.
     Let $\Gamma=(\QQ_{\geq 0})^{d_1}\times (\Q_{>0})^{d_2}$.
By Theorem~\ref{thm:fubini} we can compute $\Vol(S'_0)$ by means of the constructible integral
$$\Vol(S'_0)=\int_{\Gamma}\varphi\,d\chi'.$$
  Let $w\in \Z^{d_1}_{>0}\times \Z_{<0}^{d_2}$ be the weight vector of the $\GG_{m,k}$-action on
 $\AA^{d_1}_k\times \AA^{d_2}_k$, and let $$\nu:\Q^{d_1}\times \Q^{d_2}\to \Q^{d_1}\times \Q^{d_2-1}$$ be the projection
 in the direction of $w$ onto the product of $\Q^{d_1}$ with any coordinate hyperplane in $\Q^{d_2}$.
   Since $S'$ is stable under the $\GG_{m,k}(\Ka)$-action on $O\times Z(\Ka)$, Proposition~\ref{prop:torstable} implies that
   the function $\varphi$ is constant on every fiber of $\nu$. Moreover,
   the intersection of $\Gamma$ with any fiber of $\nu$ is either a half-open bounded line segment (if $d_1>0$) or an open half-line (if $d_1=0$).
   In both cases, its bounded Euler characteristic vanishes.  Thus
 $\Vol(S'_0)=0$ by Proposition~\ref{prop:fubini}.
  \end{proof}

\begin{rem}\label{rem:KS}
In the general statement of Kontsevich and Soibelman's integral identity conjecture, $f$ is a formal function, rather than a polynomial.
 It is claimed in \cite{thuong} that the results of Hrushovski and Kazhdan extend to the formal-rigid setting, and this claim is then used to prove
 Conjecture B when $f$ is a formal power series. Unfortunately, we have not been able to construct a proof of this claim. If it holds, then
 our methods can also be generalized to the case where $f$ is formal.
\end{rem}

\section{Further generalizations}\label{sec:general}

\subsection{Relative Grothendieck rings}\label{ssec:rel}
We can refine the preceding constructions by working relatively over a base variety instead of over $k$. Let $B$ be a Noetherian $k$-scheme
 with trivial $\gal$-action.
  When we speak of a $B$-scheme $X$ with $\gal$-action, we will always assume that the structural morphism $X\to B$ is $\gal$-equivariant.
 The Grothendieck group $\gro^{\gal}(\Var_B)$ of $B$-varieties with $\gal$-action is the abelian group characterized by the following presentation:
\begin{itemize}
\item {\em Generators}: isomorphism classes of $B$-schemes of finite type $X$ endowed with a good $\gal$-action. Here ``good'' means that the action factors through $\mu_n(k)$ for some $n>0$ and that every orbit is contained in an affine open subscheme of $X$.
  Isomorphism classes are taken with respect to $\gal$-equivariant isomorphisms over $B$.
\item {\em Relations}: we consider two types of relations.
\begin{enumerate}
\item {\em Scissor relations}: if $X$ is a $B$-scheme of finite type with a good $\gal$-action and $Y$ is a $\gal$-stable closed subscheme of $X$, then
$$[X]=[Y]+[X\setminus Y].$$
\item {\em Trivialization of linear actions}: let $X$ be a $B$-scheme with a good $\gal$-action, and let $V$ be a $k$-vector space
of dimension $d$ with a good linear action of $\gal$.
Then $$[X\times_k V]=[X\times_k \AA^d_k]$$ where the $\gal$-action on $X\times_k V$ is the diagonal action, the action on $\AA^d_k$ is trivial, and the $B$-structures are induced by the one on $X$.
\end{enumerate}
\end{itemize}
The group $\gro^{\gal}(\Var_B)$ has a unique ring structure such that $[X]\cdot [X']=[X\times_B X']$ for all $B$-schemes of finite type $X$, $X'$ with good $\gal$-action. Here the $\gal$-action
on $X\times_B X'$ is the diagonal action.
 The identity element in $\gro^{\gal}(\Var_B)$ is $[B]$, the class of the base scheme $B$.
 With a slight abuse of notation, we continue to write $\LL$ for the class of $\AA^1_B$ (with the trivial $\gal$-action) in the ring $\gro^{\gal}(\Var_B)$.

 Every morphism $p:B'\to B$ of Noetherian $k$-schemes induces a base change morphism of rings
 $$p^*:\gro^{\gal}(\Var_B)\to \gro^{\gal}(\Var_{B'}):[X]\mapsto [X\times_B B']$$
 and, if $p$ is of finite type, a pushforward morphism of groups
 $$p_{!} :\gro^{\gal}(\Var_{B'})\to \gro^{\gal}(\Var_{B}):[X]\mapsto [X]$$
 that forgets the $B'$-structure. The pullback morphism $p^*$ sends $\LL$ to $\LL$; the pushforward morphism $p_{!}$ sends $\LL$ to $[B']\cdot \LL$.

 The construction of the motivic volume can be refined to a relative setting by means of the following results.
\begin{lemma}\label{lemm:spread}
Let $B$ be a Noetherian $k$-scheme. For every point $b$ of $B$, we denote by $\kappa(b)$ the residue field of $B$ at $b$ and by $\iota_b:\Spec \kappa(b)\to B$ the inclusion map. Then the morphism
$$\iota=\prod_{b\in B}\iota_b^{\ast}:\gro^{\gal}(\Var_B)\to \prod_{b\in B} \gro^{\gal}(\Var_{\kappa(b)})  $$
is injective.
\end{lemma}
\begin{proof}
Let $b$ be a point of $B$ and let $Z$ be its Zariski closure in $B$, endowed with its reduced induced structure.
Then one can copy the proof of \cite[3.4]{NiSe-K0} to show that $\gro^{\gal}(\Var_{\kappa(b)})$ is the direct limit of the rings
$\gro^{\gal}(\Var_{U})$ where $U$ runs through any fundamental system of open neighbourhoods of $b$ in $Z$. Now the result follows from
Noetherian induction and the scissor relations in the Grothendieck ring.
\end{proof}

Let $Y$ be a $\Ka$-scheme of finite type and let $T$ be a semi-algebraic subset of $Y$.
 For every algebraically closed valued field extension $L$ of $\Ka$, the formulas that define $T$ in $Y(\Ka)$ also define a semi-algebraic subset of
 $Y(L)$, which we will denote by $T(L)$. This set does not depend on the choice of the formulas defining $T$, by quantifier elimination for algebraically closed
 valued fields.

\begin{prop}\label{prop:relvol}
Let $\cY$ be an $R_0$-scheme of finite type and let $S$ be a semi-algebraic subset of $\cY(\Ra)$ defined over $K_0$.
  For every point $y$ in $\cY_k$ we denote by $\kappa(y)$ the residue field of $\cY_k$ at $y$. We set $R_y=\kappa(y)\llb t\rrb$ and $K_y=\kappa(y)\llpar t\rrpar$ and we fix an algebraic closure $\Ka_y$ of $K_y$.
  We denote by $\iota_y$ the inclusion map $\Spec \kappa(y)\to \cY_k$.
 Then there exists a unique element $\alpha$ in $K_0^{\gal}(\Var_{\cY_k})$ such that, for every point $y$ of $\cY_k$, we have
 $$\iota_y^\ast(\alpha)=\Vol(S(\Ka_y)\cap \spe^{-1}_{\cY\times_R R_y}(y))$$ in $\gro^{\gal}(\Var_{\kappa(y)})$.
  This element satisfies
  $p_{!}(\alpha)=\Vol(S)$ in $\gro^{\gal}(\Var_k)$, where $p$ denotes the projection $\cY_k\to \Spec k$.
\end{prop}
\begin{proof}
 Rather than going through all the constructions in \cite{HK}, we will give a proof based on our computation of the motivic volume in the sch{\"o}n case (Proposition \ref{prop:schonfubini}).
 Uniqueness of $\alpha$ follows immediately from Lemma \ref{lemm:spread}, so it suffices to prove existence.

{\em Step 1: the sch{\"o}n case.}
We first prove the assertion in the following special case. Let $Y$ be a sch{\"o}n closed subvariety of $\GG^{n}_{m,K_0}$, for some $n>0$.
Set $X=Y\times_{K_0}K$ and let $\Sigma$ be a tropical fan for $X$ in $\R^n\times\R_{\geq 0}$. This fan defines a toric $R_0$-scheme $\PP_0(\Sigma)$ as well as a toric $R$-scheme $\PP(\Sigma)$; the latter is the normalization of $\PP_0(\Sigma)\times_{R_0}R$.
 Let $\cY$ be the schematic closure of $Y$ in $\PP_0(\Sigma)$, and let $\cX$ be the schematic closure of $X$ in $\PP(\Sigma)$.
 For every cell $\gamma$ in $\Sigma_1$, we denote by $\cX_k(\gamma)$ the intersection of $\cX_k$ with the torus orbit of $\PP(\Sigma)_k$ corresponding to $\gamma$.
  The natural $R_0$-morphism $\PP(\Sigma)\to \PP_0(\Sigma)$ induces a morphism of $k$-schemes $h:\cX_k\to \cY_k$.
  Let $C$ be a constructible subset of $\cY_k$ and set $S=\spe^{-1}_{\cY}(C)\cap Y(\Ka)$. We set
 $$\alpha=\sum_{\gamma\ \mathrm{bounded}}[\cX_k(\gamma)\cap h^{-1}(C)](1-\LL)^{\dim(\gamma)}$$ in $\gro^{\gal}(\Var_{\cY_k})$, where $\gamma$ runs through the set of bounded cells in $\Sigma_1$. We claim that $\alpha$ satisfies all the properties in the statement.
  To prove this, we may assume that $h^{-1}(C)$ is contained in a unique stratum $\cX_k(\gamma)$, by additivity.
 Then the claim follows immediately from the fact that the construction of the pairs $(\PP_0(\Sigma),\cY)$ and $(\PP(\Sigma),\cX)$ is compatible with extensions of the residue field $k$,
 and the semi-algebraic bijection $\psi$ constructed in the proof of Proposition \ref{prop:fubini} commutes with specialization.

{\em Step 2: the general case.}
  By additivity, we may assume that $\cY$ is a subscheme of a split $R_0$-torus $\TT$ with cocharacter lattice $N$. By Proposition \ref{prop:schon}, we can further reduce to the case
  where there exist a monomial morphism of $K_0$-tori $\varphi:T'\to T=\TT_{K_0}$, a sch{\"o}n closed subvariety $U$ of $T'$ and a constructible subset $\Gamma$ of $N'_{\QQ}$ (where $N'$ is the cocharacter lattice of $T'$) such that $\varphi$ maps $S'=U(\Ka)\cap \trop^{-1}(\Gamma)$ bijectively onto $S$.
   Let $\nu:N'\to N$ be the affine map of cocharacter lattices associated with $\varphi$. Since $S$ is contained in $\TT(\Ra)$, we know that $\trop(S')$ is contained in the affine subspace
   $A=\nu_{\Q}^{-1}(0)$ of $N'_{\Q}$. Intersecting $\Gamma$ with $A$, we may assume that $\Gamma$ is contained in $A$.

 Let $\Sigma$ be a tropical fan for $U$ in $N'_{\R}\times \R_{\geq} 0$.
 Every refinement of $\Sigma$ is still a tropical fan for $U$, so that we may assume that $A\cap \Sigma_1$ and $\Gamma$  are unions of cells in the polyhedral complex $\Sigma_1$.
  By additivity, we can then further reduce to the case where $\Gamma$ is a relatively open cell in $\Sigma_1$.
The morphism
 $\varphi:T'\to T$ extends on an $R_0$-morphism $$\widetilde{\varphi}:(\PP_0(\Sigma)\setminus D) \to \TT$$ where $D$ is the union of the irreducible components in $\PP_0(\Sigma)_k$ corresponding to vertices of $\Sigma_1$ that do not lie in $A$.
 If we denote by $\cU$ the schematic closure of $U$ in $\PP_0(\Sigma)$, then $S'$ is contained in $(\cU\setminus D)(\Ra)$ because $\trop(S')$ is contained in $A$. Moreover, we can write $S'$ as $\spe^{-1}_{\cU}(C)\cap U(\Ka)$ where $C$ is the the intersection of $\cU_k$ with the torus orbit of $\PP_0(\Sigma)_k$ corresponding to the cell $\Gamma$.

   We can apply step 1 of the proof to the $R_0$-scheme $\cU$ and the semi-algebraic set
   $S'$.
  This yields an element $\alpha'$ in
 $\gro^{\gal}(\Var_{\cU^o_k})$, where $\cU^o_k=\cU_k\setminus D$. Let $\phi:\cU^o_k\to \cY_k$ be the morphism obtained from $\widetilde{\varphi}$ by restriction. Then $\alpha=\phi_{!}(\alpha')$ satisfies all the properties in the statement: we have $$p_{!}(\alpha)=(p\circ \phi)_{!}(\alpha')=\Vol(S')=\Vol(S).$$
 Furthermore, let $y$ be a point of $\cY_k$,
 write $Z=\cU_k\times_{\cY_k}y$, and denote by $\iota_Z$ the projection morphism $Z\to \cU_k$ and by $q$ the projection $Z\to \Spec \kappa(y)$. Then
  $$\iota_{y}^\ast\phi_{!}(\alpha')=q_{!}\iota_{Z}^{\ast}(\alpha')=\Vol(S'(\Ka_y)\cap \spe^{-1}_{\cU\times_{R_0} R_y}(Z))=\Vol(S(\Ka_y)\cap \spe^{-1}_{\cY\times_{R_0} R_y}(y))$$
 in $\gro^{\gal}(\Var_{\kappa(y)})$, where the second equality again follows from step 1, applied to the $R_y$-scheme $\cV=\cU\times_{R_0}R_y$ and the semi-algebraic set
 $$S'(\Ka_y)\cap \spe^{-1}_{\cV}(Z)=\spe^{-1}_{\cV}(Z\cap (C\times_k \kappa(y)))\cap U(\Ka_y) .$$
\end{proof}

Let $\cY$ be an $R_0$-scheme of finite type and let $S$ be a semi-algebraic subset of $\cY(\Ra)$ defined over $K_0$. Then
we will continue to denote the unique object $\alpha$ in Proposition \ref{prop:relvol} by
$$\Vol(S)\in \gro^{\gal}(\Var_{\cY_k}).$$ This is a harmless abuse of notation: this object is mapped to the motivic volume $\Vol(S)\in \gro^{\gal}(\Var_k)$ from Theorem \ref{thm:HK} by the pushforward morphism $$p_{!}:\gro^{\gal}(\Var_{\cY_k})\to \gro^{\gal}(\Var_{k})$$ associated with the projection $p:\cY_k\to \Spec k$.
 If $\iota:Z\to \cY$ is an immersion, then the definition of $\Vol(S)$ implies at once that
 $$\Vol(S\cap \spe^{-1}_{\cY}(Z))=\iota_{!}\iota^\ast\Vol(S)$$ in $\gro^{\gal}(\Var_{\cY_k})$.

With these refinements at hand, we can now upgrade the comparison result with Denef and Loeser's motivic nearby fiber and prove the relative version of the Davison-Mainhardt conjecture.

\begin{prop}[Motivic volume of a strict normal crossings model]\label{prop:sncrel}
With the notations and assumptions of Theorem \ref{thm:snc}, we have
$$\Vol(\cX(\Ra))=\sum_{\emptyset \neq J \subset I}(1-\LL)^{|J|-1}[\widetilde{E}_J^o]$$ in $\gro^{\gal}(\Var_{\cX_k})$.
\end{prop}
\begin{proof}
This follows immediately from Theorem \ref{thm:snc} and the fact that the property of being a strict normal crossings model is preserved under extensions of the residue field $k$.
\end{proof}
\begin{cor}[Comparison with the motivic nearby fiber]\label{cor:DLrel}
 Let $f:U\to \Spec k[t]$ be a morphism of varieties over $k$, with smooth generic fiber,
and denote by $\cX$ the base change of $U$ from $k[t]$ to $R_0=k\llb t\rrb$.
 Then
 the image of $\Vol(\cX(\Ra))\in \gro^{\gal}(\Var_{\cX_k})$ in the localized Grothendieck ring $\gro^{\gal}(\Var_{\cX_k})[\LL^{-1}]$ is
equal to Denef and Loeser's motivic nearby fiber of $f$.
\end{cor}
\begin{proof}
This follows from a direct comparison of the formula in Proposition \ref{prop:sncrel} with Denef and Loeser's formula for the motivic nearby fiber in \cite[3.5.3]{DL}.
\end{proof}

\begin{prop}[Relative Davison-Meinhardt conjecture]\label{prop:DMrel}
With the notations and assumptions of Theorem \ref{thm:DM}, the equality $$\Vol(\cV(\Ra))=[f^{-1}(1)]$$ is valid already in $\gro^{\gal}(\Var_Y)$, where we view both sides of the
equation as objects over $Y$ {\em via} the projection $p:U\to Y$.
\end{prop}
\begin{proof}
 The special fiber $\cV_k$ is canonically isomorphic to the closed subscheme $U_0$ of $U$ defined by $f=0$.
 Let $y$ be a point of $Y$ with residue field $\kappa(y)$, and set $R_y=\kappa(y)\llb t\rrb$.
We must show that
 $$\Vol(\spe^{-1}_{\cV\times_R R_y}(U_0\times_Y y))=[f^{-1}(1)\times_{Y} y]$$
 in $\gro^{\gal}(\Var_{\kappa(y)})$. Performing a base change from $k$ to $\kappa(y)$, we can reduce to the case where $y$ is $k$-rational.
 Now one can simply copy the proof of Theorem \ref{thm:DM}, replacing $Y(\Ra)$ by $\spe^{-1}_{Y\times_k R_0}(y)$.
\end{proof}

We now state and prove further natural generalizations of Theorems~\ref{thm:DM} and \ref{thm:KS}, replacing the affine spaces on which $\GG_m$ acts with positive weights by invariant subvarieties of circle compact toric varieties with $\GG_m$-action, as explained below.  The proofs are essentially identical to those in Section~\ref{sec:appli}.  Rather than repeating the details verbatim, we briefly indicate the minor changes that need to be made in each case.

\subsection{Circle compactness} Let $F$ be a field. Following \cite{BBS}, we say that a variety $X$ over $F$ with $\GG_{m,F}$-action is \emph{circle compact} if the limit of $s \cdot x$, as $s$ goes to zero in $\GG_{m,F}$, exists for every point $x$ in $X$.  Note that circle compactness depends on the choice of the $\GG_{m,F}$-action, not just the underlying variety.  For instance, if $\GG_{m,F}$ acts on $\AA^n_F$ with weights $w_1, \ldots, w_n$, then $\AA^n_F$ is circle compact if and only if each weight $w_i$ is nonnegative.  We characterize circle compact toric varieties as follows.

\begin{lem}\label{lem:circompact} Let $\TT_F$ be a split torus over $F$ with cocharacter lattice $N$. Let $\Sigma$ be a rational polyhedral fan in $N_{\R}$ and let
$X(\Sigma)$ be the corresponding toric variety. We choose a point $w$ in $N$ and we let $\GG_{m,F}$ act on $X(\Sigma)$ via the cocharacter $\gamma_w : \GG_{m,F} \rightarrow \TT_F$.  Then the following are equivalent.
\begin{enumerate}
\item The toric variety $X(\Sigma)$ is circle compact.
\item  The support  $|\Sigma|$ of the fan $\Sigma$ is star-shaped around $w$.
\item  The support  $|\Sigma|$ of the fan $\Sigma$ contains $|\Sigma| + w$.
\end{enumerate}
\end{lem}

\begin{proof}
Note that $X(\Sigma)$ is circle compact if and only if the limit of $s \cdot x$ exists for one point $x$ in each orbit of the dense torus.  Let $x_\tau$ be a point in the orbit corresponding to a cone $\tau \in \Sigma$.  Then the limit of $s \cdot x_\tau$ exists if and only if the image of $w$ in $N/ \mathrm{span}(\tau)$ is contained in $|\Star_\Sigma(\tau)|$.

Now, if $|\Sigma|$ is star-shaped around $w$, then the image of $w$ is contained in $|\Star_\Sigma(\tau)|$ for all $\tau$.  On the other hand, if $|\Sigma|$ is not star-shaped around $w$, then there is a closed ray starting from $w$ whose intersection with $|\Sigma|$ is disconnected. Let $w'$ be the point closest to $w$ in a connected component that does not contain $w$, and let $\tau \in \Sigma$ be the cone that contains $w'$ in its relative interior.  Then the limit of $s \cdot x_\tau$ does not exist.  This shows that the first two conditions are equivalent.

We now show that the second and third conditions are equivalent.  First, if $|\Sigma|$ is star shaped around $w$, then it is a union of convex cones that contain $w$, and hence it contains $|\Sigma| + w$.  On the other hand, suppose $|\Sigma|$ contains $|\Sigma| + w$ and let $w'$ be any point in $|\Sigma|$.  Then $|\Sigma|$ contains $rw'$ for all positive real numbers $r$, and hence it contains $rw' + w$.  Rescaling again shows that $|\Sigma|$ contains $\frac{rw' + w}{r+1}$ for all positive real numbers $r$, and hence it contains the open interval $(w,w')$.  This shows that $|\Sigma|$ is star shaped around $w$, and therefore the second and third conditions are equivalent, as claimed.
\end{proof}

\subsection{A generalization of the Theorem~\ref{thm:DM}}

We generalize~\ref{thm:DM}, replacing the affine space $\AA^d_k$ on which $\GG_{m,k}$ acts with positive weights by a $\GG_{m,k}$-invariant subvariety of a circle compact toric variety.

\begin{thm} \label{thm:genDM}
Let $X(\Sigma)$ be a toric variety over $k$, with $\GG_{m,k}$ acting by a cocharacter  $\gamma_w : \GG_{m,k} \rightarrow \TT_k$ of the dense torus $\TT_k$.  Suppose that $X(\Sigma)$ is circle compact, and let $$X \subset X(\Sigma)$$ be a $\GG_{m,k}$-invariant subvariety.

Let $Y$ be a $k$-variety endowed with the trivial $\GG_{m,k}$-action, let $U= X\times_k Y$, and let $$f:U\to \AA^1_k=\Spec k[t]$$ be a $\GG_{m,k}$-equivariant function, where $\GG_{m,k}$ acts on $\AA^1_k$ with weight $d>0$.
 We set $\cV=U\times_{k[t]}R_0$, and we endow $f^{-1}(1)$ with the $\gal$-action that factors through $\mu_d(k)$ and is given by
      $$\mu_d(k)\times f^{-1}(1)\to f^{-1}(1):(\zeta,(x,y))\mapsto (\gamma_w(\zeta) x,y).$$
            We view $\cV_k$ and $f^{-1}(1)$ as $Y$-schemes {\em via} the projection $U\to Y$.
 Then $$\Vol(\cV(\Ra))=[f^{-1}(1)]$$ in $\gro^{\gal}(\Var_Y)$
\end{thm}

\begin{proof}
The proof is similar to that of Theorem~\ref{thm:DM} and Proposition~\ref{prop:DMrel}, using the fact that $|\Sigma|$ is star-shaped around $w/d$ and contains $|\Sigma|+w/d$, by Lemma~\ref{lem:circompact}, to show that the relevant constructible subsets of $N_{\Q}$ have bounded Euler characteristic zero.
\end{proof}

\subsection{A generalization of Theorem~\ref{thm:KS}}

As in the generalization of Theorem~\ref{thm:DM}, we replace the affine space $\AA^{d_1}_k$ on which $\GG_{m,k}$ acts with negative weights by a $\GG_{m,k}$-invariant subvariety of a circle compact toric variety.
 Let $X(\Sigma)$ be a toric variety over $k$, with $\GG_{m,k}$ acting by a cocharacter $\gamma_w : \GG_{m,k} \rightarrow \TT_k$ of the dense torus.  Suppose $X(\Sigma)$ is circle compact, and let $X \subset X(\Sigma)$ be a connected $\GG_{m,k}$-invariant closed subvariety, with $x_0 \in X$ a closed point.

We also replace the affine space $\AA^{d_2}_k$ on which $\GG_{m,k}$ acts with negative weights by a $\GG_{m,k}$-invariant subvariety of an affine toric variety with repelling fixed point, as follows.
 Let $U_{\sigma'}$ be the affine toric variety over $k$ corresponding to a rational polyhedral cone $\sigma'$, equipped with the $\GG_{m,k}$-action given by a cocharacter $\gamma_{w'}$, where $-w'$ is a lattice point in the interior of $\sigma'$.  Let $X' \subset U_{\sigma'}$ be a $\GG_{m,k}$-invariant closed subvariety.  Since $-w'$ is in the interior of $\sigma'$, the fixed point
\[
x_0' = \lim_{s' \rightarrow \infty} \gamma_{w'}(s').
\]
is repelling, and hence $X'$ is connected and contains $x_0'$.

\begin{rem}
The statement we are going to prove depends only on a $\GG_{m,k}$-invariant affine neighborhood of the repelling fixed point, so there is no loss of generality in assuming that this factor is affine.
\end{rem}

With the notation above, we have the following generalization of Theorem~\ref{thm:KS}.

\begin{thm}  \label{thm:genKS}

Let $Z$ be a variety over $k$, equipped with the trivial $\GG_{m,k}$-action, and let $p$ be a closed point on $Z$.  Let $U = X \times X' \times Z$, and let
\[
f: U \rightarrow \AA^1_k
\]
be a $\GG_{m,k}$-invariant function such that $f(x_0,x_0',p) = 0$.

We view $X$ and $Z$ as closed subvarieties of $U$ via the embeddings $x\mapsto (x,x_0',p)$ and $z\mapsto (x_0,x'_0,z)$, respectively, and set $\cV=U\times_{k[t]}R_0$  and $\cW=Z\times_{k[t]}R_0$.
 Then $f$ vanishes on $X$, and
$$\Vol(\spe^{-1}_{\cV}(X))= [X] \Vol(\spe^{-1}_{\cW}(p))$$
 in $\gro^{\gal}(\Var_k)$.
\end{thm}

\begin{proof}
The proof is similar to that of Theorem~\ref{thm:KS}, using the fact that $|\Sigma|$ is star shaped around $w$, by Lemma~\ref{lem:circompact}, and that $\sigma'$ is convex with $-w'$ in its interior to show that the non-empty intersections of $|\Sigma| \times \mathring{\sigma}'$ with lines parallel to $(w,w')$ are half-open intervals or open rays.
  \end{proof}

\end{document}